\documentclass[review,3p]{elsarticle}

\journal{Computer Methods in Applied Mechanics and Engineering}

\usepackage{color}

\definecolor{dgreen}{rgb}{0,0.6,0.1}

\definecolor{dred}{rgb}{0.8,0,0}

\definecolor{dmagenta}{rgb}{0.6,0,0.6}

\definecolor{dmag}{rgb}{0.8,0,0.9}

\definecolor{dblue}{rgb}{0,0,0.7}

\definecolor{dbrown}{rgb}{0.8,0.25,0.25}

\definecolor{doran}{rgb}{1,0.4,0}

\usepackage{todonotes}

\usepackage[percent]{overpic}
\usepackage{relsize}
\usepackage{multirow}
\usepackage{siunitx}

\usepackage{amsmath,amssymb,amsthm}
\usepackage{graphicx}
\usepackage{multirow}

\newtheorem{thm}{Theorem}[section]
\newtheorem{prop}[thm]{Proposition}

\newtheorem{remark}{Remark}

\newcommand\Om{{\Omega}}
\newcommand\E{{\PP}}

\newcommand\bgrad{\textrm{\bf grad}}
\renewcommand\div{{\rm div}}
\newcommand\bcurl{\textrm{\bf curl}}
\newcommand\curl{\textrm{\bf curl}}
\newcommand\brot{\textrm{\bf rot}}
\newcommand\rot{{\rm rot}}

\newcommand{\dO}{\,{\rm d}\Omega}
\newcommand{\dE}{\,{\rm d}\E}
\newcommand{\dPP}{\,{\rm d}\PP}
\newcommand{\ds}{\,{\rm d}s}
\newcommand{\dS}{\,{\rm d}S}
\newcommand{\de}{\,{\rm d}e}
\newcommand{\df}{\,{\rm d}f}

\newcommand\nn{\boldsymbol n}
\renewcommand\tt{\boldsymbol t}

\newcommand\xxP{{\boldsymbol x}_{\PP}}
\newcommand\xxf{{\boldsymbol x}_{f}}
\newcommand\qq{\boldsymbol q}

\newcommand\bpsi{\boldsymbol \psi}
\newcommand\bphi{\boldsymbol \varphi}

\newcommand\vv{\boldsymbol v}

\renewcommand\qq{\boldsymbol q}
\newcommand\pp{\boldsymbol p}

\newcommand\HH{\boldsymbol H}
\newcommand\BB{\boldsymbol B}
\newcommand\jj{\boldsymbol j}

\newcommand\xx{\boldsymbol x}

\newcommand\ww{\boldsymbol w}
\renewcommand{\O}{ {\mathcal O}}

\newcommand\Th{{\mathcal T}_h}

\newcommand\R{\mathbb{R}}
\renewcommand{\P}{ {\mathbb P}}

\newcommand\PP{{\text P}}
\newcommand\tbn{|\!|\!|}
 \numberwithin{equation}{section}
 
\newcommand{\nodal}{\rm{node}}
\newcommand{\edge}{\rm{edge}}
\newcommand{\face}{\rm{face}}

\newcommand{\RT}{RT_0}
\newcommand{\ND}{N_0}

\definecolor{dgreen}{rgb}{0,0.6,0.1}

\definecolor{dred}{rgb}{0.8,0,0}

\definecolor{dmagenta}{rgb}{0.6,0,0.6}

\definecolor{dmag}{rgb}{0.8,0,0.9}

\definecolor{dblue}{rgb}{0,0,0.7}

\definecolor{dbrown}{rgb}{0.8,0.25,0.25}

\definecolor{doran}{rgb}{1,0.4,0}

\begin{document}
% ------------------------------------------------------------------------------------------------------------
\begin{frontmatter}

\title{Lowest order Virtual Element approximation of magnetostatic problems}

\author[unimib,imati]{L. Beir\~ao da Veiga}
\ead{lourenco.beirao@unimib.it}

\author[imati]{F. Brezzi\corref{correspondingauthor}}
\ead{brezzi@imati.cnr.it}

\author[unimib]{F. Dassi}
\ead{franco.dassi@unimib.it}

\author[unipv,imati]{L.D. Marini}
\ead{marini@imati.cnr.it}

\author[unimib,imati]{A. Russo}
\ead{alessandro.russo@unimib.it}

\address[unimib]{Dipartimento di Matematica e Applicazioni, Universit\`a di 
Milano--Bicocca, Via Cozzi 53, I-20153, Milano, Italy}
\address[imati]{IMATI CNR, Via Ferrata 1, I-27100 Pavia, Italy}
\address[unipv]{Dipartimento di Matematica, Universit\`a di Pavia,
Via Ferrata 5, I-27100 Pavia, Italy}
\cortext[correspondingauthor]{Corresponding author}

\begin{abstract}We give here a simplified presentation of the lowest order Serendipity Virtual Element method, and show its use for the numerical solution of linear magneto-static problems in three dimensions.  The method can be applied to very general decompositions of the computational domain (as is natural for Virtual Element Methods) and uses as unknowns the (constant) tangential component of the magnetic field $\HH$ on each edge, and the vertex values of the Lagrange multiplier $p$ (used to enforce the solenoidality of the magnetic induction $\BB=\mu\HH$).  In this respect the method  can be seen as the natural generalization of the lowest order Edge Finite Element Method (the so-called ``first kind N\'ed\'elec'' elements) to polyhedra of almost arbitrary shape, and as we show on some numerical examples it  exhibits very good accuracy (for being a lowest order element) and excellent robustness with respect to distortions.
\end{abstract}


\end{frontmatter}

\section{Introduction} In this paper we introduce a simplified version of  the Serendipity Virtual Element Methods
(VEMs)  presented in \cite{SERE-nod} and \cite{SERE-mix} and we show 
how they can be used for the numerical solution of linear magneto-static problems in the so-called Kikuchi formulation (see e.g. \cite{KikuchiIEEE}).  

Serendipity VEMs are a recent variant of Virtual Element Methods that allow (as is the case of classical Serendipity Finite Elements (FEMs) on quadrilaterals and hexahedra) to eliminate a certain number of  degrees of freedom (internal to faces and volumes) without compromising the order of accuracy.  In the Virtual Element framework they are particularly useful since the original formulations of VEMs (as in \cite{volley} or \cite{super-misti}) often use more degrees of freedom than their FEM counterpart (when it exists). 

The advantage of VEMs, when it comes to Serendipity variants, is that, contrary to FEMs, they do not use a reference element: an inevitable sacrifice, if you want to be able to deal with very general geometries. Such a sacrifice, that requires additional computations on the current
element, has however the advantage of  being much more robust with respect to distortions, whereas Serendipity FEMs can lose orders of accuracy already on innocent quadrilaterals 
that are not parallelograms (as is well known, and has been analyzed e.g. in \cite{A-B-F-nod}, \cite{A-B-F-Quads}).

Here, as we said,  we present a variant of the general theories of \cite{SERE-nod} and \cite{SERE-mix}, that is  specially designed for {\it lowest order} cases and comes out to be simpler, both for the theoretical presentation and the practical implementation.

 Then we apply  it to  the classical {\it model} magnetostatic problem, in a {\it smooth-enough} {simply connected} bounded  domain $\Om$ in $\R^3$:
\begin{equation}
\left\{
\begin{aligned}\label{Max3}
& \mbox{given }\jj \in H(\div;\Om) \quad(\mbox{with }\div \jj =0 \mbox{ in }\Om),~\mbox{ and }\mu=\mu(\xx)\ge \mu_0>0,\\
& \mbox {find  }\HH\in H(\bcurl;\Om) \mbox{ and }\BB\in H(\div;\Om) \mbox{ such that: }\\
& \bcurl \HH=\jj  \mbox{ and }\div\BB=0, \mbox{ with }\BB=\mu\HH \mbox{ in }\Om\\
&\mbox{with the boundary conditions } \HH\wedge\nn =0\mbox{ on }\partial\Om . 
\end{aligned}
\right.
\end{equation}

In particular we shall deal with the 
variational formulation introduced  in
\cite{KikuchiIEEE}, that reads
\begin{equation}\label{K1_3}
\left\{
\begin{aligned}
%& \mbox{given }\jj \in H(\div;\Om) \quad(\mbox{with }\div \jj =0 \mbox{ in }\Om),\quad\mbox{ and }\mu\in\R,\\
& \mbox {find  }\HH\in H_0(\bcurl;\Om) \mbox{ and }p\in H^1_0(\Om) \mbox{ such that: }\\
& \int_{\Om}\bcurl\HH\cdot\bcurl\vv\dO+\int_{\Om}\nabla p\cdot\mu\vv\dO=\int_{\Om}\jj\cdot\bcurl\vv\dO
\quad\forall\vv\in H_0(\bcurl;\Om)\\
& \int_{\Om}\nabla q\cdot\mu\HH\dO=0\quad\forall q\in H^1_0(\Om).\\ 
\end{aligned}
\right.
\end{equation}

For many other different approaches to the same problem see e.g. \cite{Monk-Maxw}, \cite{DemkoDG-Max-2}, \cite{Bermudez-Book} and the references therein.

In our discretization, the scalar variable $p$ (Lagrange multiplier for the condition $\div(\mu\HH)=0$) will be discretized using only vertex values as degrees of freedom, and the magnetic field $\HH$ will be discretized using only one degree of freedom (= constant tangential component)  per edge.  In its turn the current $\jj$ (here a given quantity)  will be discretized by its  lowest order {\it Face Virtual Element interpolant ${\jj}_I$}, individuated by its constant normal component on each face.

On tetrahedrons  this would correspond to use a piecewise linear scalar for $p$, a lowest-order 
N\'ed\'elec  of  the {\it first kind} for $\HH$, and a lowest order Raviart-Thomas for $\jj$: in a sense, nothing new. But already on prisms, pyramids, or hexahedra we start gaining, as we can allow  more general geometries and more dramatic distortions, and there are no difficulties in using much more general polyhedrons.
 
 On polyhedrons the present approach could also be seen as  being close to
previous works on Mimetic Finite Differences (the ancestor of Virtual Elements) like \cite{Brezzi:Buffa:2010} or \cite{Maxwell-MFD}. Here however the approach is more simple and direct, allowing a 
 thorough  analysis of convergence properties. Also the use of an explicit stabilizing term, reminiscent of Hybrid Discontinuous Galerkin methods (see e.g.  \cite{Cockburn-DiPietro-Ern} and the references therein) contributes, in our opinion, to the user-friendliness of the presentation.

A layout of the paper is as follows. The next section will be dedicated to recall the basic notation of functional spaces and differential operators.

Then in Section \ref{2D}  we will introduce and discuss the two-dimensional VEMs ({\it nodal} and {\it edge}) that will be used on the faces of the three-dimensional decompositions. As usual, we will present first the spaces on a single two dimensional element ({\it local spaces}).

In Section \ref{3D} we will finally present our ``Simplified Serendipity Spaces''  in three dimensions. We first deal with a single element (polyhedron) and then discuss the 
spaces on a general decomposition.

In Section \ref{MS} we will use these spaces to discretize the linear magneto-static problem, and briefly discuss their convergence and the a-priori error analysis.

Finally, in Section \ref{numres} we will present some numerical results.

\section{Notation}\label{NOT}

 In any dimension, for an integer $s\ge -1$ we will denote by $\P_s$ the space of polynomials of degree $\le s$.
	Following a common convention, $\P_{-1}\equiv\{0\}$ and $\P_0\equiv\R$. Moreover, {$\Pi_{s,\O}$ will denote the $L^2(\O)$-orthogonal projection onto $\P_s$ (or $(\P_s)^2$, or  $(\P_s)^3$). When no confusion can occur, this will be simply denoted by $\Pi_s$.}

\subsection{Basic notation in 2 and 3 dimensions}

In two and three dimensions we will denote by $\xx$  the indipendent variable. 
In two dimensions we will also use $\xx=(x,y)$
or (more often) $\xx=(x_1,x_2)$ following the circumstances. We will also use
\begin{equation}\label{xperp}
\xx^{\perp}:=(-x_2,x_1)
\end{equation}
In two dimensions, for a vector $\vv$ and a scalar $q$ we will write 
\begin{equation}\label{rotebrot}
\displaystyle{\rot \vv:=\frac{\partial v_2}{\partial x}-\frac{\partial v_1}{\partial y}},
\qquad \brot \,q:=(\frac{\partial q}{\partial y}, -\frac{\partial q}{\partial x})^T .
\end{equation}
We observe that
\begin{equation}\label{seserve}
\div (p_0 \xx)=\rot (p_0 \xx^{\perp})=2 p_0,\qquad \forall p_0\in \R .
\end{equation}
The following decompositions of polynomial vector spaces are well known and will be useful in what follows. In two dimensions we have, for all $s\ge 0$:
\begin{equation}\label{decoPs}
(\P_s)^2=\brot \P_{s+1}\oplus \xx \P_{s-1} \quad\mbox{ and }\quad(\P_s)^2=\bgrad \P_{s+1}\oplus \xx^{\perp} \P_{s-1}.
\end{equation}
%
%implying that 
% \begin{equation}\label{isodiv}
%\mbox{   $\div$ is an isomorphism between  $\Big\{\xx\P_s\}$ and $\P_s$}
%\end{equation}
% \begin{equation}\label{isorot}
%\mbox{   $\rot$ is an isomorphism between  $\Big\{\xx^{\perp}\P_s\}$ and $\P_s$}.
%\end{equation}
In three dimensions the analogues of \eqref{decoPs} (always for all $s\ge 0$) are
\begin{equation}\label{decompPs3D}
(\P_s)^3=\bcurl((\P_{s+1})^3) \oplus\xx\P_{s-1},\quad\mbox{ and }\quad
%\end{equation}
%\begin{equation}\label{decompPk3gxw}
(\P_s)^3=\bgrad(\P_{s+1}) \oplus\xx\wedge(\P_{s-1})^3.
\end{equation}
We also note that by direct computation we have, {similarly to  \eqref{seserve}}: 
 \begin{equation}\label{isodiv-3}
{\div (p_0 \xx)=3 p_0\quad \forall  p_0 \in \P_0 \qquad\mbox{ and}\qquad\bcurl ( {\boldsymbol p}_0\wedge\xx)=2 {\boldsymbol p}_0}\quad \forall {\boldsymbol p}_0 \in (\P_0)^3.
\end{equation}
Finally, on a polyhedron $\PP$ we set $\xx_{\PP} = \xx-{\bf b}_{\PP}$ with ${\bf b}_P$ the barycenter of $\PP$. Analogously, for each face $f \in \partial \PP$, we set $\xxf=\xx-{\bf b}_{f}$, with ${\bf b}_{f}=$ barycenter of $f$. Note that clearly 
\begin{equation}\label{xxpmed0}
\int_{\PP}\xx_{\PP}\dPP=0,
\end{equation}
as well as
\begin{equation}\label{xxfmed0}
\int_{f}\xx_f\df=0 .
\end{equation}

\subsection{Polynomial spaces: Raviart-Thomas and N\'ed\'elec}
We recall the definition of the classical lowest order Raviart-Thomas {\it local} spaces in $d$ space dimensions
\begin{equation}\label{RT}
\RT = (\P_0)^d \oplus \xx \: \P_0 ,
\end{equation}
and also the classical lowest order N\'ed\'elec first-type {\it local} spaces in two and three space dimensions
\begin{equation}\label{ND}
\ND = (\P_0)^2 \oplus \xx^\perp \: \P_0 
\quad \textrm{ or } \quad
\ND = (\P_0)^3 \oplus \xx\wedge(\P_{0})^3 .
\end{equation}
In what follows, when dealing with the {\it faces} of a polyhedron (or of a polyhedral decomposition) we shall use two-dimensional differential operators that act on the restrictions to faces of scalar functions that are defined on a three-dimensional domain. Similarly, for vector valued functions we will use two-dimensional differential operators that act on the restrictions to faces of the tangential components.  In many cases, no confusion will be likely to occur; however, to stay on the safe side, we will often use a superscript $\tau$  to denote the tangential components of a three-dimensional vector, and a subscript $f$ to indicate the two-dimensional differential operator. Hence, to fix ideas, if a face has equation
$x_3=0$ then $\xx^{\tau}:=(x_1,x_2)$ and, say, $\div_f\vv^{\tau}:=\frac{\partial v_1}{\partial x_1}+
\frac{\partial v_2}{\partial x_2}$.

\subsection{Some Functional Spaces}

We recall some commonly used functional spaces: 
\begin{align*}
&H(\div;\Om)=\{\vv\in [L^2(\Om)]^3 \mbox{ with } \div \vv \in L^2(\Om)\},\\ 
&H_0(\div;\Om)=\{\bphi\in H(\div;\Om)\mbox { with } \bphi\cdot\nn=0 \mbox{ on }\partial\Om\},\\ 
&H(\curl;\Om)=\{\vv\in [L^2(\Om)]^3 \mbox{ with } \curl \vv \in [L^2(\Om)]^3\},\\ 
&H_0(\curl;\Om)=\{\vv \in H(\curl;\Om) \mbox{ with } \HH\wedge\nn =0 \mbox{ on } \partial \Om \},\\
&H^1(\Om)=\{q\in L^2(\Om) \mbox{ with } \bgrad q \in (L^2(\Om))^2\},\\
&H^1_0(\Om)=\{q\in H^1(\Om) \mbox{ with } q=0 \mbox{ on } \partial\Om\}.
\end{align*}

%\subsection{Variational formulation of the Magnetostatic problem}

%We recall that, for a vector-valued function $\HH\in H(\bcurl;\Om)$, 
%the boundary  condition $\HH\wedge\nn =0\mbox{ on }\partial\Om$ is usually expressed as $\HH\in H_0(\bcurl;\Om)$.
%\medskip

%Here we will deal with the variational formulation introduced  in
%\cite{KikuchiIEEE}, that reads
%
%
%\begin{equation}\label{K1_3}
%\left\{
%\begin{aligned}
%& \mbox {find  }\HH\in H_0(\bcurl;\Om) \mbox{ and }p\in H^1_0(\Om) \mbox{ such that: }\\
%& \int_{\Om}\bcurl\HH\cdot\bcurl\vv\dO+\int_{\Om}\nabla p\cdot\mu\vv\dO=\int_{\Om}\jj\cdot\bcurl\vv\dO
%\quad\forall\vv\in H_0(\bcurl;\Om)\\
%& \int_{\Om}\nabla q\cdot\mu\HH\dO=0\quad\forall q\in H^1_0(\Om).\\ 
%\end{aligned}
%\right.
%\end{equation}
%It is easy to check, by the usual theory of mixed methods, that \eqref{K1_3} has a unique solution $(\HH,p)$. Then we check that
%$\HH$ and $\mu\HH$  give the solution of  \eqref{Max3}  and $p=0$.  Checking  that $p=0$ is immediate, just taking $\vv=\nabla p$ in the first equation. Once we know that $p=0$, the first equation gives $\bcurl\HH=\jj$, and then the second equation gives $\div\mu\HH=0$.

% ------------------------------------------------------------------------------------
\section{Two-dimensional Simplified Serendipity Spaces}\label{2D}
% ------------------------------------------------------------------------------------
 
We begin with the definition of the local spaces. Let $\PP$ denote a generic polyhedron, and let $f$ be a face of such polyhedron. For the time being, we only assume that all faces $f$ are simply connected.

%%%%%%%%%%%%%%%%%%%
\renewcommand{\k}{0}
\newcommand{\kd}{\k}
\newcommand{\kr}{\k}
%%%%%%%%%%%%%%%%%%%

% ------------------------------------------------------------------------------------
\subsection{The  local spaces on faces}
% ------------------------------------------------------------------------------------

The following spaces are a simplified version of the local lowest-order {\it nodal} and {\it edge} Serendipity spaces on faces introduced in \cite{max2}.

\subsection{ The local {\it nodal} space on faces}

The  present local {\it nodal} Virtual Element space on faces can be seen as an extension to polygons
of the simplest space of piecewise linear functions on triangles, as well as a simplified version of the
basic VEM nodal spaces.

Indeed, in \cite{SERE-nod} (formula (2.5) with $k=1,\,k_{\Delta}=0$)
the basic lowest order local nodal space on faces was introduced as
\begin{equation}\label{0-nodf-1}
\widetilde{V}_{1}^{\nodal}(f)  := \Big\{ q \in C^0(f): \: q_{|e} \in \P_{1}(e) \: \forall e \in \partial f, \: \Delta q \in \P_{0}(f)\Big\}. 
 \end{equation}
It is clear that we always have $\P_1\subset\widetilde{V}_{1}^{\nodal}(f)$. The  original degrees of freedom proposed in \cite{volley} or in \cite{projectors} were  
%{\begin{itemize}
%\item The values of $q$ at the vertices:
%\begin{equation}\label{0-2dofnv}
% \mbox{ for each vertex $\nu$, the value $q(\nu)$. } 
%\end{equation}
%\item The mean value of $q$ on $f$
%\begin{equation}\label{0-2dofnmv}
%\frac{1}{|f|}\int_f q\df.
%\end{equation}
%\end{itemize}}
%
%\begin{equation*}
\begin{align}
&\bullet \mbox{ The values of }q \mbox{ at the vertices \qquad (that is : for each vertex }\nu, \mbox{ the value }q(\nu));\label{0-2dofnv} \\ 
%\end{equation}
%\begin{equation}
&\bullet \mbox{ The mean value of }q \mbox{ on } f\qquad \mbox{(that is: for each face } f,\quad \frac{1}{|f|}\int_f q\df).\label{0-2dofnmv}
\end{align}
%\end{equation*}
%
However, it was pointed out in \cite{max2} that the degrees of freedom \eqref{0-2dofnmv} could be replaced by the integral
 \begin{equation}\label{altdof}
 \int_f \nabla q\cdot \xx_f\df.
 \end{equation}
Indeed, the obvious identity 
 \begin{equation}\label{0-nodf-2}
2\int_f q\df=\int_f q\,\div\xx_f\df=-\int_f \nabla q\cdot \xx_f\df+\int_{\partial f}q\xx_f\cdot\nn\ds 
 \end{equation}
shows that given boundary values and mean value \eqref{0-2dofnmv} on $f$ one can compute \eqref{altdof} and conversely, given the boundary values 
and \eqref{altdof}  one can compute the mean value on $f$. 
 
 At this point,  remembering  \eqref{xxfmed0}, we  easily see that 
 \begin{equation}\label{zerosuq1}
 \int_f \nabla q\cdot \xx_f\df=0\qquad\forall q\in \P_1,
 \end{equation} 
so that the linear subspace of  $\widetilde{V}_{1}^{\nodal}(f)$ made of those $q$ that verify  \eqref{zerosuq1} will still contain $\P_1$.  Such a subspace can obviously be written as
\begin{equation}\label{0-nodf}
V_{1}^{\nodal}(f)    := \Big\{ q\in C^0(f) : q_{|e} \in \P_{1}(e) \: \forall e \in \partial f, \: \Delta q \in \P_{\kd}(f), \int_{f}\nabla q \cdot \xxf \df=0\Big\}, 
 \end{equation}
 and will be our local {\it scalar nodal VEM space}, with the degrees of freedom \eqref{0-2dofnv}
(i.e., the vertex values).
% -------
{ Then from the unisolvence of the degrees of freedom \eqref{0-2dofnv} and \eqref{0-2dofnmv} in $\widetilde{V}_{1}^{\nodal}(f)$ it follows rather obviously  that the degrees of freedom \eqref{0-2dofnv} are unisolvent in $V_{1}^{\nodal}(f)$.}
% (since the integral \eqref{altdof} is always zero by the definition of the space).}
Indeed: for a $q\in V_{1}^{\nodal}(f)$,  the degrees of freedom \eqref{0-2dofnv} give the value of $q$ along the whole $\partial f$.
Then, following \eqref{0-nodf-2}, and using the information that $ \int_{f}\nabla q \cdot \xxf \df=0$ (that is included in the definition \eqref{0-nodf} of $q\in V_{1}^{\nodal}(f)$) we have 
 \begin{equation}\label{0-nodf-2-0}
2\int_f q\df=\int_{\partial f}q\,\xx_f\cdot\nn\ds ,
 \end{equation}
and then from the boundary values of $q$ we can easily compute the mean value of $q$ on $f$. Now we know both the degrees of freedom
\eqref{0-2dofnv}  {\it and} \eqref{0-2dofnmv}, and we are back on the track of the space
$\widetilde{V}_{1}^{\nodal}(f)$.

We point out that, if needed, out of the  d.o.f. in \eqref{0-2dofnv} we can compute the $L^2$-projection $\Pi_1$ of $\nabla q$ onto $(\P_1)^2$. Indeed, from the definition of projection and \eqref{0-nodf-2}
%\eqref{equiv-dof} 
we deduce
\begin{equation}\label{proj-nodal-2D}
\int_f \Pi_1 \nabla q\cdot \pp_1\, \df:=\int_f  \nabla q\cdot \pp_1\, \df= -\int_f q\, \div \pp_1 \df +\int_{\partial f} q\, \pp_1\cdot \nn \ds,
\end{equation}
and the last two terms are computable for every $\pp_1\in(\P_1)^2$.

\subsection{ The local {\it edge} space on faces}
{
The  present local {\it edge} Virtual Element space on faces can be seen as an extension to polygons
of the simplest space of lowest order  N\'ed\'elec elements of the first kind $\ND$  (see \eqref{ND}), as well as a simplified version of the
basic VEM edge spaces in \cite{max2,SERE-mix}.

Indeed, in \cite{SERE-mix} (formula (6.1) with $k=0,\, k_d=0, \, k_r=0$)
the basic lowest order local edge space on faces was introduced as

\begin{equation}\label{0-edgef-1}
\widetilde{V}_{0}^{\edge}(f)  := \Big\{ \vv \in (L^2(f))^2\,:\div\vv \in \P_{0}(f), \: \rot\vv \in \P_{0}(f), \vv \cdot \tt_e \in \P_{0}(e) \: \forall e \in \partial f\\
\Big\} ,
\end{equation}
 with the degrees of freedom given by the (constant) tangential components on each edge, plus
the integral
\begin{equation}\label{edge1}
\int _f\vv \cdot \xx_f\df.
\end{equation}
  
The space given in \eqref{0-edgef-1} clearly contains all constant vector fields.
%, as well as
%vector fields of the form $\vv=\pp_0+\xx_f ^\perp p_0$ with $\pp_0\in\R^2$ and $p_0\in\R$ (that is, the lowest order N\'ed\'elec elements of the first kind, as we have seen). 
However, for these functions the
degree of freedom \eqref{edge1} is identically zero (due to \eqref{xxfmed0}).
%(since $ \xx_f ^\perp\cdot\xx_f\equiv 0$, and the mean value of $\xx_f$ on $f$ is also zero). 
%Hence our target (to keep the space of the  lowest order N\'ed\'elec elements of the first kind inside the space)  will be equally satisfied in the subspace
Hence we could consider the subspace
\begin{equation} \label{0-edgef}
\begin{aligned}
V^{\edge}_{0}(f) := \Big\{ \vv \in (L^2(f))^2\,:\,&\div\vv \in \P_{\kd}(f), \: \rot\vv \in \P_{\kr}(f), \vv \cdot \tt_e \in \P_{\k}(e) \: \forall e \in \partial f,\\
&\mbox{ and }~ \int_{f}\vv\cdot\xxf \df=0\Big\} ,
\end{aligned}
\end{equation}
that we will take, from now on, as our edge VEM space, and have that it always contains 
the constant vector fields.
%with dimension
%\begin{equation}\label{dimVef}
%\mbox{dim } V^{\rm e}_{k-1}(f) = k N_e+ 2 \pi_{k-1,2}-1. %+k(k+1)-1.
%\end{equation}
In $V^{\edge}_{\k}(f)$ we will therefore have the degrees of freedom: 
\begin{equation}\label{0-2dofe1}
\bullet\quad\mbox{ on each $e\in \partial f$, the moments $\displaystyle\int_e \vv\cdot\tt_e \de $.}
\end{equation}
Note that the number of degrees of freedom of $V^{\nodal}_{1}(f)$ (=number of vertexes) and of  $V^{\edge}_{\k}(f)$  (=number of edges) obviously coincide.

\begin{remark} It is immediate to check that the space  \eqref{0-edgef} {contains, together with
constant vector fields,  also all} vector fields of the form $\vv=\pp_0+\xx_f ^\perp p_0$ with $\pp_0\in\R^2$ and $p_0\in\R$ (that is, the lowest order N\'ed\'elec elements of the first kind, {that we recalled} in \eqref{ND}).\hfill\qed
\end{remark}

 We observe that the d.o.f. \eqref{0-2dofe1} allow to compute, for each $\vv\in V^{\edge}_{\k}(f)$, the (constant) value of $\rot \vv$ on the face $f$ by the usual Stokes theorem
 \begin{equation}\label{Stokesf}
 |f| \rot\vv=\int_f\rot\vv\df=\int_{\partial f}\vv\cdot\tt\ds ,
\end{equation} 
 as well as the $L^2$-orthogonal projection $\Pi_1: V^{\edge}_{0}(f) \rightarrow (\P_1(f))^2$. Indeed, using \eqref{decoPs} and integrating by parts we have
\begin{equation}\label{proj-edge-2D}
\begin{aligned}
\int_f \Pi_1 \vv \cdot \pp_1 \df&:= \int_f  \vv \cdot \pp_1 \df = \int_f  \vv \cdot (\brot\, p_2+\xxf p_0) \df\\
& =\int_f \rot \,\vv \,p_2 \df + \int_{\partial f} \vv\cdot \tt \,p_2 \ds + \int_f \vv\cdot \xxf p_0 \df\\
& =\int_f \rot \,\vv \,p_2 \df + \int_{\partial f} \vv\cdot \tt \,p_2 \ds,
\end{aligned}
\end{equation}
and all the terms in the right-hand side are computable.

%{\color{blue}
%{\subsection{A computable scalar product in $V^{\edge}_{0}(f)$}
%Having the possibility to compute, for each $\vv\in V^{\edge}_{0}(f)$, its $(L^2(f))^2$ projection on 
%$(\P_1(f))^2$ we can then define (and compute!) the {\it discrete scalar product} between two elements $\vv$ and $\ww$ of $V^{\edge}_{0}(f)$
%\begin{equation}\label{scal-e-f}
%[\vv,\ww]_{{\edge}, f}:=\int_f (\Pi_1\vv)\cdot(\Pi_1\ww)\df+h_f\int_{\partial f}
%[(\vv-\Pi_1\vv)\cdot\tt]\,[(\ww-\Pi_1\ww)\cdot\tt]\ds
%\end{equation}
%where $h_f$ is the diameter of the face $f$. 
%
% Moreover, under the mesh assumptions introduced in Section \ref{theo:est}, it is (lenghty but) not difficult to check that  the above scalar product is equivalent to the original one in the sense
%that 
%\begin{equation}
%\alpha_* \int_f\vv\cdot\ww\df\le [\vv,\ww]_{{\edge},f}\le\alpha^*\int_f\vv\cdot\ww\df
%\end{equation}
%with $\alpha_*$ and $\alpha^*$ positive constants.
%}
%}

We close this section with a simple but important result.

\begin{prop}\label{comm-2D}
It holds
\begin{equation}\label{ntoe2}
\nabla V^{\nodal}_1(f)=\{ \vv \in V^{\edge}_0(f):~ \rot \vv =0\}.
\end{equation}
\end{prop}
\begin{proof}
We start by noting that, for any function $q \in V^{\nodal}_1(f)$, it holds 
$$   
\nabla q \in V^{\edge}_0(f)~ \mbox{ and }~ \rot \nabla q =0.
$$
Indeed, it is immediate to check that $\nabla q$ satisfies all the requirements in the definition of $V^{\edge}_0(f)$ and, being a gradient, it also has vanishing rotor. Therefore 
$$\nabla V^{\nodal}_1(f) \subseteq \{ \vv \in V^{\edge}_0(f):~ \rot \vv =0\},$$ that combined with 
$$
\begin{aligned}
\textrm{dim} \left( \nabla V^{\nodal}_1(f) \right) &= \textrm{dim} \left( V^{\nodal}_1(f) \right)  - 1 
= \textrm{dim} \left( V^{\edge}_0(f) \right) - 1  \\
& = \textrm{dim} \Big( \{ \vv \in V^{\edge}_0(f):~ \rot \vv =0\} \Big)
\end{aligned}
$$
yields the proof.
\end{proof}

\section{Three-dimensional spaces}\label{3D}

Let  $\PP$ be a polyhedron. For the time being, we just assume that $\PP$ and all of its faces are simply connected. Let $N_v$ be the number of vertices, $N_e$ the number of edges, and $N_{f}$ the number of faces of $\PP$.

% ---------------------------------------------------------------------------------
\newcommand{\kdP}{k-1}
\subsection{ The local spaces on polyhedrons}
% ---------------------------------------------------------------------------------

For each face $f$ we are going to use the spaces $V^{\nodal}_{1}(f)$ and $V^{\edge}_{\k}(f)$ as defined in \eqref{0-nodf} and \eqref{0-edgef}, respectively.  Then we introduce their three-dimensional  analogues.
% In order to define  the corresponding spaces on $\PP$ we first define the {\it nodal Serendipity boundary spaces} as
% \begin{equation}\label{aux1}
% \calB^n_S(\partial\PP):=\{q\in C^0(\partial\PP)~\mbox{ such that } q_{|f}\in SV^n_{1}(f)\quad\forall \mbox{ face }f\in\partial\PP\},
%\end{equation}
%and  the {\it edge Serendipity boundary spaces} as  
% \begin{multline}\label{aux2}
% \calB^e_{S}(\partial\PP):=\{\vv \mbox{ such that } \vv_{|f}\in SV^e_{\k}(f)\quad\forall \mbox{ face }f\in\partial\PP, \\
% \mbox{ and }
% \vv\cdot\tt_e \mbox{ continuous at each edge } e\in\partial\PP\}.
%\end{multline}

% ---------------------------------------------------------------------------------
\subsubsection{ The local {\it nodal} spaces}
% ---------------------------------------------------------------------------------

The three dimensional local nodal space is defined as 
\begin{equation}\label{nod3}
 V^{\nodal}_{1}(\PP)    := \Big\{ q \in C^0(\PP) \: : q_{|f}\in V^{\nodal}_{1}(f)~\forall \mbox{ face }f\in\partial\PP, \:  \,\Delta\,q = 0 \mbox{ in } \PP\Big\}. 
\end{equation}
In $V^{\nodal}_{1}(\PP)$ the degrees of freedom are simply
\begin{equation}
 \bullet \mbox{ for each vertex $\nu$, the nodal value $q(\nu).$ } \label{dof-3dnk-1}
\end{equation}
%
%\begin{prop}\label{unisolv-VnP}
%The operators \eqref{dof-3dnk-1} constitute a set of degrees of freedom for $V^n_{1}(\PP)$.
%\end{prop}
%\begin{proof}
%Using Proposition \ref{unisolv-Vnf} for all faces $f \in \partial P$, it follows immediately that \eqref{dof-3dnk-1} constitute a set of degrees of freedom for the boundary space $V^n_{1}(\PP)|_{\partial P}$. Since the functions in $V^n_{1}(\PP)$ are harmonic, the boundary space is isomorphic to the full space $V^n_{1}(\PP)$, and the proof follows.
%\end{proof}

\noindent We note that
$$
\P_1(\PP) \subseteq V^{\nodal}_{1}(\PP) ,
$$
since first order polynomials clearly satisfy all the conditions in \eqref{nod3}. 

%the definition of $V^{\nodal}_{1}(\PP)$.

% ---------------------------------------------------------------------------------
\subsubsection{The local {\it edge} spaces}
% ---------------------------------------------------------------------------------

{In analogy with the two-dimensional case we start by recalling  the three dimensional edge space  defined in \cite{SERE-mix} as 
\begin{multline} \label{edge3}
\widetilde{V}^{\edge}_{\k}(\PP) := \Big\{ \vv\in(L^2(\PP))^3\,: {\vv^{\tau}}_{|f}\in V^{\edge}_{\k}(f)~\forall \mbox{ face }f\in\partial\PP,~\vv\cdot\tt_e \mbox{ continuous at each edge } e\in\partial\PP, \\
\div\vv=0 \mbox{ in }\PP, \: \curl(\curl\vv) \in (\P_{\kr}(\PP))^3
 \Big\} ,
\end{multline}
with the degrees of freedom given by the values of the (constant) tangential components on each
edge plus the integrals
\begin{equation}\label{pippo3}
\int_{\PP}(\bcurl\vv)\cdot(\xx_{\PP}\wedge {\pp}_{\kr})\,\dPP \quad \forall\,\pp_{\kr}\in(\P_0(\PP))^3.
\end{equation}
We observe, first, that the constant vectors 
%belonging to the  three-dimensional lowest order N\'ed\'elec space of first kind
%(namely: $(\P_0)^3+\xx_{\PP}\wedge(\P_0)^3$, as defined in \eqref{ND}) 
are contained in $\widetilde{V}^{\edge}_{\k}(\PP)$, and for them the integral in \eqref{pippo3}
is always equal to zero}. Hence, following the path that is becoming usual here, we set}
\begin{multline} \label{edge3-bis}
 V^{\edge}_{\k}(\PP) := \Big\{ \vv\in(L^2(\PP))^3\,: {\vv^{\tau}}_{|f}\in V^{\edge}_{\k}(f)~\forall \mbox{ face }f\in\partial\PP,~\vv\cdot\tt_e \mbox{ continuous at each edge } e\in\partial\PP, \\
\div\vv=0 \mbox{ in }\PP, \: \curl(\curl\vv) \in (\P_{\kr}(\PP))^3,
~\int_{\PP}(\bcurl\vv)\cdot(\xx_{\PP}\wedge {\pp}_{\kr})\,\dPP =0\quad \forall {\pp}_{\kr} \in (\P_{\kr}(\PP))^3 \Big\} ,
\end{multline}
{and observe that all the constant vector fields 
 are contained in ${V}^{\edge}_{\k}(\PP)$}.

\begin{remark} \label{AllN13D}It is easy to check that the space  \eqref{edge3-bis}, together with
constant vector fields, contains all vector fields of the form $\vv=\pp_0+\xx_{\PP}\wedge  \qq_0$ with $\pp_0$ and $\qq_0$ in $\R^3$ (that is, the lowest order N\'ed\'elec elements of the first kind in three dimensions, as defined in \eqref{ND}).\hfill\qed
\end{remark}

In $V^{\edge}_{\k}(\PP) $ we have now the degrees of freedom,
\begin{equation}
 \bullet\mbox{ on each edge $e \in \partial P$ the moments $\displaystyle\int_e \vv\cdot\tt_e \de .$ } \label{dof-3dek-1}
\end{equation}

Out of the above degrees of freedom we can compute the $(L^2(\PP))^3$-orthogonal projection $\Pi_{0}$ from $V^{\edge}_{\k}(\PP)$ to $(\P_0(\PP))^3$. 
Indeed,
 by definition of projection,  \eqref{isodiv-3}, and an integration by parts we have:
\begin{equation}\label{proj-edge-3D}
\begin{aligned}
\int_{\PP} \Pi_0\vv \cdot \pp_0 \dPP&:=\int_{\PP} \vv \cdot \pp_0 \dPP=\int_{\PP} \vv \cdot \curl(\xxP \wedge{\boldsymbol q}_0) \dPP ~(\mbox{for }{\boldsymbol q}_0 = -\frac{1}{2} {\boldsymbol p}_0)\\
&= 
%\underbrace
{\int_{\PP} \curl\vv \cdot (\xxP \wedge{\boldsymbol q}_0) \dPP} + \int_{\partial \PP} (\vv\wedge \nn) \cdot (\xxP \wedge{\boldsymbol q}_0) \dS\\
&=\hskip2cm 0~~\qquad  \qquad \!\!\!\!+\int_{\partial \PP}  \Big(\nn \wedge(\xxP \wedge{\boldsymbol q}_0)\Big) \cdot \vv\dS\\
%\underbrace
&=%\hskip2cm 0~~\qquad  \qquad \!\!\!\!+
%\underbrace
{\sum_f \int_f \Big(\nn \wedge(\xxP \wedge{\boldsymbol q}_0)\Big)^{\tau} \cdot {\vv^{\tau}}  \df,}
\end{aligned}
\end{equation}
that is computable as in \eqref{proj-edge-2D}.
%{\color{red} Ci devo meditare
%
%Following \cite{SERE-mix} we see that, out of the above degrees of freedom, we can compute the $(L^2(\PP))^3$-orthogonal projection $\Pi_{0}$ from $V^{\rm{e}}_{\k}(\PP)$ to $(\P_0(\PP))^3$. Indeed,  writing $V^{\rm{e}}_{\k}$ in \eqref{edge3} as 
%$V^{\rm{e}}_{\{(\beta,\beta_d,\beta_r),k_d,(\mu,\mu_d\mu_r)\}}$
% (that is, with the notation of \cite{SERE-mix}) we know that we can take
%$s=\min\{\beta_d,k_d+1,\mu_r\}$, that here are all equal to $\k$.}
%
Hence, we can define a scalar product
\begin{equation}\label{PSe3k}
[\vv,\ww]_{{\edge},\PP}:=(\Pi_{\k}\vv,\Pi_{\k}\ww)_{0,\PP}+
h^2_{\PP} \sum_{e\in \partial \PP} \int_e [(\vv-\Pi_0\vv)\cdot\tt_e]\,[(\ww-\Pi_0\ww)\cdot\tt_e]\de
\end{equation}
and we note that (assuming  very mild  mesh regularity conditions, for instance as in Section \ref{theo:est}) we have
\begin{equation}\label{PSe3k-1}
\alpha_* (\vv,\vv)_{0,\PP}\le [\vv,\vv]_{{\edge},\PP}\le 
\alpha^* (\vv,\vv)_{0,\PP}\qquad\forall \vv\in V^{\edge}_{\k}(\PP)
\end{equation}
{for some constants $\alpha_*, \alpha^*$ independent of $h_{\PP}$}.
We observe that
\begin{equation}\label{consiE3k}
[\vv,{\pp}_{\k}]_{{\edge},\PP}=\int_{\PP}\vv\cdot {\pp}_{\k}\dE =(\vv,{\pp}_{\k})_{0,\PP}\quad\forall \vv\in V^{\edge}_{\k}(\PP),\;\forall {\pp}_{\k}\in (\P_{\k}(\PP))^3.
\end{equation}
\subsubsection{The local {\it face} spaces}
% ---------------------------------------------------------------------------------

 In three dimensions we will also need a Virtual Element face space.  For it, we proceed as in the previous case. We start with the space defined in \cite{SERE-mix} 
 \begin{equation} \label{face3}
\widetilde{V}^{\face}_{\k}(\PP) \!:=\! \Big\{ \bpsi\in (L^2(\PP))^3: \bpsi\cdot\nn_f\in\P_{\kr}(f)~\forall\mbox{ face }f,\,\div\bpsi\!\in\! \P_{\k}(\PP), \, \curl\bpsi\!\in\!\! (\P_{\kr}(\PP))^3
 \Big\},
\end{equation}
where the degrees of freedom are given by the (constant) values of the normal components on the faces plus the value of the integrals:
\begin{equation}\label{pippof}
\int_{\PP}
\bpsi\cdot (\xx_{\PP}\wedge {\pp}_{\kr}) \dPP \quad \forall {\pp}_{\kr} \in [\P_{\kr}(\PP)]^3.
\end{equation}
We note that the constant vector fields are inside this space, but also that the value of the integral in \eqref{pippof},
 for $\bpsi$ constant, is always equal to zero, due to \eqref{xxpmed0}, so that we can define 
\begin{multline} \label{face3-bis}
\! {V}^{\face}_{\k}(\PP) \!:=\! \Big\{ \bpsi\in (L^2(\PP))^3: \bpsi\cdot\nn_f\in\P_{\kr}(f)~\forall\mbox{ face }f,\,\div\bpsi\!\in\! \P_{\k}(\PP), \, \curl\bpsi\!\in\!\! (\P_{\kr}(\PP))^3, \\
	\mbox{and}\int_{\PP}
\bpsi\cdot (\xx_{\PP}\wedge {\pp}_{\kr}) \dPP =0\quad \forall {\pp}_{\kr} \in [\P_{\kr}(\PP)]^3
  \Big\}.
\end{multline}

\begin{remark}\label{AllRT3D} It is easy to check that the space  \eqref{face3-bis}, together with
constant vector fields, contains all vector fields of the form 
  $(\P_0)^3+\xx_{\PP}\P_0$ (that is, the lowest order Raviart-Thomas space, as defined 
  in \eqref{RT}).  \qed
\end{remark}

\noindent For $V^{\face}_{\k}(\PP)$ we have  therefore the degrees of freedom
\begin{equation}\label{dof-3dfk-1}
\bullet\mbox{ for each face $f \in \partial P$ the moments $\displaystyle\int_f
\bpsi\cdot\nn_f \df . $} 
\end{equation}
{Clearly, out of the degrees of freedom \eqref{dof-3dfk-1} we can easily compute
the (constant) value of $\div\bpsi$
\begin{equation}\label{div}
\div\bpsi=\frac{1}{|\PP|}\int_{\PP}\div\bpsi\dPP=\frac{1}{|\PP|}\int_{\partial\PP}\bpsi\cdot\nn\dS .
\end{equation}
}
According to \cite{SERE-mix} we {also have, now, }that from the above degrees of freedom we can compute
the $(L^2(\PP))^3$-orthogonal projection %$\Pi_{s}$ 
from $V^{\face}_{\k}(\PP)$ to $(\P_0(\PP))^3$
(and, actually, to $(\P_1(\PP))^3$). Indeed, using \eqref{decompPs3D}, {an integration by parts, and \eqref{face3-bis}, } we have:
\begin{equation}\label{proj-face-3D}
\begin{aligned}
\int_{\PP} \Pi_{1} \bpsi \cdot \pp_1 \dPP&: =\int_{\PP} \bpsi \cdot \pp_1 \dPP =\int_{\PP}  \bpsi \cdot (\nabla p_2 + \xx_{\PP} \wedge \pp_0) \dPP\\
&= -\int_{\PP} \div \bpsi \, p_2 \dPP + \int_{\partial \PP} \bpsi \cdot \nn \,p_2\dS + \int_{\PP}  \bpsi \cdot  (\xx_{\PP} \wedge \pp_0) \dPP\\
&= -\int_{\PP} \div \bpsi \, p_2 \dPP + \int_{\partial \PP} \bpsi \cdot \nn \,p_2\dS,
\end{aligned}
\end{equation}
{where, using \eqref{div} and \eqref{dof-3dfk-1}}, all the terms in the right-hand side are computable.
 Hence, we can define
a scalar product
\begin{equation}\label{PSf3k}
[\bpsi,\bphi]_{{\face},\PP}:=(\Pi_{\k}\bpsi,\Pi_{\k}\bphi)_{0,\PP}+
{h_{\PP}\sum_{f\in \partial \PP} \int_f[(I-\Pi_{\k})\bpsi\cdot\nn][(I-\Pi_{\k})\bphi\cdot\nn]\df},
\end{equation}
and we again note that, assuming  very mild  mesh regularity conditions, for instance as in Section \ref{theo:est},
{
\begin{equation}\label{PSf3k-1}
\alpha_1 (\bpsi,\bpsi)_{0,\PP}\le [\bpsi,\bpsi]_{{\face},\PP}\le 
\alpha_2 (\bpsi,\bpsi)_{0,\PP}\qquad\forall \bpsi\in V^{\face}_{\k}(\PP) 
\end{equation}
for positive constants $\alpha_1, \alpha_2$ independent of $h_{\PP}$.}
Moreover, we also have
\begin{equation}\label{consiF3k}
[\bpsi,{\bf p}_{\k}]_{{\face},\PP}=\int_{\PP}\bpsi\cdot{\pp}_{\k}\dE=(\bpsi, {\pp}_{\k})_{0,\PP}\qquad\forall \bpsi\in V^{\face}_{\k}(\PP),\;\forall{\pp}_{\k}\in (\P_{\k}(\PP))^3.
\end{equation}
%

%We note that using the degrees of freedom \eqref{dof-3dek-1}-\eqref{dof-3dek-3}, for every face $f$ we can compute the $(L^2(f))^2$-orthogonal projection operator from $V^{\rm{e}}_{k-1}(f)$ to $(\P_{k-1}(f))^2$, and using this and \eqref{dof-3dek-4}-\eqref{dof-3dek-4} we can compute the $(L^2(\PP))^3$-orthogonal projection from $V^{\rm{e}}_{k-1}(\PP)$ to $(\PP_{k-1}(\PP))^3$: first, integrating by parts and using the boundary data, we get the integral of $\vv$ against any polynomial vector of degree $\le k-2$ having zero divergence (and hence against the $\bcurl$ of any polynomial vector of degree $\le k-1$). Then we recall that every vector polynomial of degree
%
%%
 \begin{remark}\label{tetra} Using Remarks \ref{AllRT3D} and \ref{AllN13D}, by a simple dimensional count we see that the present local Nodal, Edge, and Face Virtual Element spaces coincide, whenever the element $\PP$ is a tetrahedron,  with the classical $\P_1$, $N_0$ and $RT_0$ elements (respectively). With a minor additional effort we could see that the same is  true when $\PP$ is a ''rectangular box''. Note that the methods do not coincide, due to a different choice of the scalar products. However, the two types of scalar products are equivalent, coincide on constants, and in practice the results are not significantly different.   Hence, the present setting might be considered as a sort of ``natural'' extension of the $\P_1$-$N_0$-$RT_0$ approach to (much) more general element geometries.
 \qed
 \end{remark}

% ---------------------------------------------------------------------------------
\subsubsection{Exact sequence properties}
% ---------------------------------------------------------------------------------

{Here we} present two important results. We first have

\begin{prop}
It holds
\begin{equation}\label{diagramVn-Ve}
\nabla V^{\nodal}_{1}(\textup\PP)=\{ \vv \in V^{\edge}_{\k}(\textup\PP):~ \bcurl \vv =0\}.
\end{equation}
\end{prop}
\begin{proof}
We first point out that, with the above definitions and using Proposition \ref{comm-2D}, for every $q\in V^{\nodal}_{1}(\PP)$, we easily have that the {\it tangential gradient} (applied {\it face by face}) will belong to $V^{\edge}_{\k}(f)$. Moreover, from the definition of $V^{\nodal}_{1}(\PP)$, we immediately have that $\bcurl \nabla q = 0$ and $\div\nabla q = \Delta q = 0$. Therefore, for any $q\in V^{\nodal}_{1}(\PP)$, it holds 
$
\nabla q \in \{ \vv \in V^{\edge}_{\k}(\PP):~ \bcurl \vv =0\}
%\nabla q \in V^e_{\k}(\PP) , \quad \bcurl \nabla q = 0 .
$
and thus  
$$
\nabla V^{\nodal}_{1}(\PP) \subseteq \{ \vv \in V^{\edge}_{\k}(\PP):~ \bcurl \vv =0\}.
$$

%{\color{red} FORSE NON NE VALE LA PENA ED E' MEGLIO IL VECCHIO? IL VECCHIO E' ORA COMMENTATO...
%
%We show the equality of the two sets by a dimensional count. It is immediate that
%\begin{equation}\label{vdim}
%{\rm dim}(\nabla V^{\nodal}_{1}(\PP)) = N_v -1 .
%\end{equation}
%We now calculate the dimension of $\{ \vv \in V^{\edge}_{\k}(\PP):~ \bcurl \vv =0\}$. We observe that $\bcurl\vv = 0$
%implies, for all $f \in \partial\PP$, 
%\begin{equation}\label{f-conds}
%0 = \int_f \bcurl\vv \cdot  \nn_f \df = \int_f \rot \vv_f \df = \sum_{e \in \partial f} \int_e \vv\cdot\tt_e^f \de ,
%\end{equation}
%where the $\tt_e^f =  \pm \tt_e$, ${e \in \partial f}$, in accordance with $\nn_f$. Recalling \eqref{dof-3dek-1} and noting that (for any sufficiently regular vector field {\bf w} living on $\PP$) 
%$$
%\sum_{f \in \partial E} \int_f \rot {\bf w}_f \df = 0 ,
%$$
%easily yield that \eqref{f-conds} constitute a set of $N_f-1$ independent conditions. 
%Therefore, also using the Euler formula for polyhedrons $N_v -N_e + N_f = 2$,
%\begin{equation}\label{edim}
%\begin{aligned}
%& {\rm dim} \{ \vv \in V^{\edge}_{\k}(\PP):~ \bcurl \vv =0\} \\
%& \le {\rm dim}(V^{\edge}_{\k}(\PP)) - (N_f - 1) = N_e -N_f + 1 = N_v -1 .
%\end{aligned}
%\end{equation}
%The proof clearly follows by combining \eqref{edim}  and \eqref{vdim}.
%}

Conversely, if $\vv \in V^{\edge}_{\k}(\PP)$ with $\bcurl \vv =0$, then $\vv = \nabla q$ for some $q \in H^1(\PP)$. Since, for any face $f \in \partial P$, it holds $\rot_f (\vv_{|f}) = \bcurl \vv \cdot \nn_f = 0$, using Proposition \ref{comm-2D} yields that $q$ restricted to the boundary of $\PP$ belongs to 
$V^{\nodal}_{1}(\PP)_{|\partial \PP}$. Finally, since $\div \vv =0$ for all $\vv \in V^{\edge}_{\k}(\PP)$, we have $\Delta q =0$. Thus $q$ belongs to $V^{\nodal}_{1}(\PP)$.
\end{proof}

\begin{prop}
It holds
\begin{equation}\label{diagramVe-Vf}
\bcurl \,V^{\edge}_{\k}(\textup\PP) := \{ \bpsi \in V^{\face}_{\k}(\textup\PP):~ \div \bpsi=0\}.
\end{equation}
\end{prop}
\begin{proof}
For every $\vv\in V^{\edge}_{\k}(\PP)$ we have that $\bpsi:=\bcurl \vv$ belongs to $V^{\face}_{\k}(\PP)$. Indeed, on each face $f$ we have that $\bpsi\cdot\nn_f(\equiv\rot_f\vv_{|f})$, from \eqref{0-edgef}
belongs to $\P_{\kr}(f)$ (as required in \eqref{face3-bis} );  moreover $\div\bpsi=0$ (obviously)
and $\bcurl\bpsi\in(\P_{\kr}(\PP))^3$ from \eqref{edge3-bis}. Hence,
\begin{equation}\label{diagramVe-Vf-b}
\bcurl \,V^{\edge}_{\k}(\PP) \subseteq \{ \bpsi \in V^{\face}_{\k}(\PP):~ \div \bpsi=0\}.
\end{equation}
We show the equality of the two sets by a dimensional count. It is immediate that the condition $\div \bpsi=0$, for $\bpsi \in V^{\face}_{\k}(\PP)$, is equivalent to 
$$
\sum_{f \in \partial \PP} \int_f \bpsi\cdot\nn_f \df = 0
$$
and thus, recalling \eqref{dof-3dfk-1},
\begin{equation}\label{fdim}
{\rm dim} \{ \bpsi \in V^{\face}_{\k}(\PP):~ \div \bpsi=0\} = N_f - 1 .
\end{equation}
By classical properties of linear operators,  using \eqref{diagramVn-Ve} and finally the Euler formula on polyhedrons, we get now
\begin{equation}\label{dybala}
\begin{aligned}
{\rm dim} (\bcurl (V^{\edge}_{\k}(\PP))) & =
{\rm dim} (V^{\edge}_{\k}(\PP)) - {\rm dim}  \{ \bpsi \in V^{\edge}_{\k}(\PP):~ \bcurl \bpsi =0\} \\ 
&=N_e- {\rm dim}(\nabla V^{node}_1)
 = N_e - (N_v -1) = N_f - 1{\color{red},}
\end{aligned}
\end{equation}
{and the} result follows from \eqref{fdim} and \eqref{dybala}.

%{\color{magenta} Conversely, given $\bpsi\in V^{\face}_{0}(\PP)$ with $\div\bpsi=0$ we first observe that taking its normal component on each face we have obviously an element  of $V^{disc}_0(\partial\PP)$  with zero mean value over $\partial\PP$ (from Gauss theorem
% and $\div\bpsi=0$). Then, proceeding as in Theorem 4 of \cite{super-misti}, from \eqref{etof2g} we have that there exists 
% an element $\vv^*\in V^{\edge}_0(\partial\PP)$ such that its face-by-face $\rot_f$ is equal to
% $\bpsi\cdot\nn$.  Taking the tangential components $\vv^*\cdot\tt$ on each edge we have
% from \eqref{dof-3dek-1} that there exists a unique $\vv\in V^{\edge}_0(\PP)$ such that
% $\vv\cdot\tt \equiv \vv^*\cdot\tt$ on each edge. In particular we will have that the tangential 
% components of $\vv$ will be equal to those of $\vv^*$ on each face. The $\bcurl$ of such $\vv$ will be an element of $V^{\face}_0(\PP)$ as we have seen already. But the (constant) normal components
% of $\bcurl\vv$ on each face are equal (Stokes theorem) to $\rot_f\vv$, and then coincide with $\rot_f\vv^*$ which is equal to $\bpsi\cdot\nn$, so that, finally, $\bcurl\vv\equiv\bpsi$ 
% (as they are two elements in $V^{\face}_0(\PP)$ with the same normal components).
%}
\end{proof}

% ---------------------------------------------------------------------------------
\subsection{The global spaces}
% ---------------------------------------------------------------------------------

Let $\Th$ be a decomposition of the computational domain $\Om$ 
into polyhedrons $\PP$.  Again, for the time being, we assume just that all polyhedrons and all their faces are simply connected. More detailed assumptions will be presented in Section \ref{theo:est}.
%,with the usual regularity assumptions
 Here we assume that 
 \begin{equation}\label{micost}
 \mbox{\it the permeability $\mu$ is constant on each $\PP$}.
 \end{equation}
  We define the {\it global spaces} as follows.
\begin{equation}
V^{\nodal}_{1}\equiv V^{\nodal}_{1}(\Om):=\Big\{ q \in H^1_0(\Om) \mbox{ such that } q_{|\PP} \in V^{\nodal}_{1}(\PP)\: \forall \PP \in \Th \Big\}, 
\end{equation}
 with the obvious degrees of freedom
\begin{equation}
\bullet \mbox{ for each vertex $\nu$: the nodal value $q(\nu).$ } \label{dof-3dnk-1G}\\
\end{equation}
For the global edge space we have
\begin{equation}
\!\!V^{\edge}_{0}\equiv V^{\edge}_{0}(\Om):=\Big\{ \vv \in H_0(\bcurl;\Om) \mbox{ such that } \vv_{|\PP} \in V^{\edge}_{0}(\PP)\: \forall \PP \in \Th \Big\},
\end{equation}
with the obvious degrees of freedom
\begin{equation}
 \bullet \mbox{ on each edge $e$: $\displaystyle\int_e \vv\cdot\tt_e \de .$ } \label{dof-3dek-1G}
\end{equation}
 Finally, for the face space we have: 
\begin{equation}\label{globf}
V^{\face}_{0}\equiv V^{\face}_{0}(\Om):=\Big\{\bpsi\in H_0(\div;\Om)\mbox{ such that }\bpsi_{|\PP}\in V^{\face}_{0}(\PP)\:\forall \PP \in \Th\Big\},
\end{equation}
with the degrees of freedom
\begin{equation}
 \bullet \mbox{ for each face $f$: $\displaystyle\int_f
\bpsi\cdot\nn_f  \df  $.} \label{dof-3dfk-1G}
\end{equation}
It is important to point out that
\begin{equation}\label{inclu3nek}
\nabla V^{\nodal}_{1}\subseteq V^{\edge}_{0}.
\end{equation}
In particular, recalling the local results \eqref{diagramVn-Ve}, it is easy to check that
\begin{equation}\label{rot03k}
\nabla V^{\nodal}_{1}\equiv\{ \vv\in V^{\edge}_{\k}\mbox { such that }\bcurl\vv=0\} .
\end{equation}
Similarly, we have
\begin{equation}\label{inclu3efk}
\bcurl \,V^{\edge}_{\k}\subseteq V^{\face}_{\k},
\end{equation}
and it can be checked (recalling the local results \eqref{diagramVe-Vf-b})
% and with an argument similar to Theorem 4 of \cite{super-misti}) 
that
\begin{equation}\label{rot03fk}
\bcurl \,V^{\edge}_{\k}\equiv\{ \bpsi\in V^{\face}_{\k}\mbox { such that }\div\bpsi=0\} .
\end{equation}
%
% PEZZO SULLE COSTANTI, LO HO TOLTO  
%
%{\color{blue} Introducing the additional space 
%\begin{equation}\label{pwc3k}
%V^c_{\k}:=\{\gamma\in L^2(\Om)\;\mbox{ such that }\gamma_{|\PP}\in \P_{\k}\;\forall\PP\in\Th\},
%\end{equation}
%%
% we also have
%\begin{equation}\label{inclu3fvk}
%\div V^{\rm{f}}_{\k}\equiv V^c_{\k}.
%\end{equation}
%}
%
\begin{remark}\label{catenak} We point out that the inclusions  \eqref{inclu3nek}, and \eqref{inclu3efk}  are
(in a sense) also {\bf practical}, and not only theoretical. By this, more specifically, we mean that: given the degrees of freedom
of a $q\in V^{\nodal}_{1}$ we can compute the corresponding degrees of freedom of $\nabla q$ in $V^{\edge}_{\k}$; and given the degrees of freedom of a $\vv\in V^{\edge}_{\k}$ we can compute the corresponding degrees of freedom of $\bcurl\,\vv$ in $V^{\face}_{\k}$. 
\hfill\qed
%
% PEZZO SULLE COSTANTI, CHE HO TOLTO
%
% finally (and this is almost obvious) from the degrees of freedom of a $\bphi\in V^{\rm{f}}_{\k}$ we can compute its divergence in each element and obtain an element in $V^c_{\k}$.\qed
%
\end{remark}
%Indeed, the inclusion $\rot V^{\rm{e}}\subseteq V^c$ is obvious. Let us see the converse. Given a $\gamma\in V^c$ we can easily find a $\ww\in (H^1_0(\Omega))^2$ such that $\rot\ww=\gamma$.
%Then we define $\ww_I\in V^{\rm{e}}$ through
%\begin{equation}\label{definterp}
%\int_{e}(\ww-\ww_I)\cdot\tt_e\de=0\qquad\forall \mbox{ edge   $e$ in }\Th
%\end{equation}
%and, if necessary (meaning: if we kept the internal degrees of freedom), 
%\begin{equation}\label{definterp0}
%\int_{\PP}(\ww-\ww_I)\cdot\xx\dE=0 \mbox{ on each element }\PP\in\Th.
%\end{equation}
% It is immediate to see that
%\begin{equation}\label{rotinterp}
%\int_{\PP}\rot(\ww-\ww_I)\dPP=\sum_{e\in\partial\E}\int_{e}(\ww-\ww_I)\cdot\tt_e\de\qquad\forall \PP\in\Th
%\end{equation}
%which implies that $\rot\ww_I=\gamma$.
%
\noindent  
From \eqref{PSe3k} we can  also define a global scalar product: 
\begin{equation}\label{PSe3gk}
[\vv,\ww]_{\edge}:=\sum_{\PP\in\Th} [\vv,\ww]_{{\edge},\PP}.
\end{equation}

We note that, recalling \eqref{PSe3k-1}, 
\begin{equation}\label{SP3boundsk}
\alpha_* (\vv,\vv)_{0,\Om}\le [ \vv,\vv]_{\edge}\le  \alpha^* (\vv,\vv)_{0,\Om}\qquad\forall \vv\in V^{\edge}_{\k} .
\end{equation}
It is also important to point out that, using \eqref{consiE3k} we have
\begin{equation}\label{consiE3gk}
[\vv,\pp]_{\edge}=( \vv,\pp)_{0,\Om}:=\int_{\Om} \vv\cdot\pp\dO \quad\forall \vv\in V^{\edge}_{\k}, \,\forall\pp\mbox{ piecewise in $(\P_{\k})^3$} .
\end{equation}

\noindent From \eqref{PSf3k} we can  also define a scalar product in $V^{\face}_{\k}$ in the obvious way
\begin{equation}\label{PSf3gk}
[\bpsi,\bphi]_{\face}:=\sum_{\PP\in\Th}[\bpsi,\bphi]_{{\face},\PP}
\end{equation}
and we note that 
{
\begin{equation}\label{SPf3boundsk}
\alpha_1 (\bpsi,\bpsi)_{0,\Om}\le [\bpsi,\bpsi]_{\face}\le \alpha_2 (\bpsi,\bpsi)_{0,\Om}\qquad\forall \bpsi\in V^{\face}_{\k}.
\end{equation}
}
It is also important to point out that, using \eqref{consiF3k} we have
\begin{equation}\label{consif3k}
[\bpsi,\pp]_{\face}=(\bpsi,\pp)_{0,\Om}:=\int_{\Om}\bpsi\cdot\pp\dO \qquad\forall \bpsi\in V^{\face}_{\k}, \,\forall\pp \mbox{ piecewise in $(\P_{\k})^3$} .
\end{equation}

%%%%%%%%%%%%%%%%%%%%%%%%%%%%%%%%%%%%%%%
%%%%%%%%%%%%%%%%%%%%%%%%%%%%%%%%%%%%%%%
%%%%%%%%%%%%%%%%%%%%%%%%%%%%%%%%%%%%%%%
%%%%%%%%%%%%%%%%%%%%%%%%%%%%%%%%%%%%%%%
%%%%%%%%%%%%%%%%%%%%%%%%%%%%%%%%%%%%%%%
%pipo

\section{Discretization of the Magneto-static Problem}\label{MS}

We are now ready to present the discretization of our Magneto-static Problem \eqref{K1_3}
that we recall here:
\begin{equation}\label{K1_3-dadi}
\left\{
\begin{aligned}
& \mbox{given }\jj \in H(\div;\Om) \quad(\mbox{with }\div \jj =0 \mbox{ in }\Om),\quad\mbox{ and }\mu=\mu(\xx)\ge \mu_0>0,\\
& \mbox {find  }\HH\in H_0(\bcurl;\Om) \mbox{ and }p\in H^1_0(\Om) \mbox{ such that: }\\
& \int_{\Om}\bcurl\HH\cdot\bcurl\vv\dO+\int_{\Om}\nabla p\cdot\mu\vv\dO=\int_{\Om}\jj\cdot\bcurl\vv\dO
\quad\forall\vv\in H_0(\bcurl;\Om)\\
& \int_{\Om}\nabla q\cdot\mu\HH\dO=0\quad\forall q\in H^1_0(\Om).\\ 
\end{aligned}
\right.
\end{equation}
It is easy to check, by the usual theory of mixed methods, that \eqref{K1_3-dadi} has a unique solution $(\HH,p)$. Then we check that
$\HH$ and $\mu\HH$  give the solution of  \eqref{Max3}  and $p=0$.  Checking  that $p=0$ is immediate, just taking $\vv=\nabla p$ in the first equation. Once we know that $p=0$, the first equation gives $\bcurl\HH=\jj$, and then the second equation gives $\div\mu\HH=0$.

\begin{remark}
We observe that an alternative variational formulation could be 
\begin{equation}\label{K1_3-Hgrad}
\left\{
\begin{aligned}
& \mbox{given }\jj \in H(\div;\Om) \quad(\mbox{with }\div \jj =0 \mbox{ in }\Om),\quad\mbox{ and }\mu=\mu(\xx)\ge \mu_0>0,\\
& \mbox {find  }\HH\in H_0(\bcurl;\Om) \mbox{ and }p\in H^1_0(\Om) \mbox{ such that: }\\
& \int_{\Om}\bcurl\HH\cdot\bcurl\vv\dO+\int_{\Om}\nabla p\cdot\mu\vv\dO=\int_{\Om}\jj\cdot\bcurl\vv\dO
\quad\forall\vv\in H_0(\bcurl;\Om)\\
& \int_{\Om}\nabla q\cdot\mu\HH\dO \,{-} \int_{\Om} \nabla p \cdot \nabla q \dO=0\quad\forall q\in H^1_0(\Om),\\ 
\end{aligned}
\right.
\end{equation}
{or other possible variants mimicking, one way or another, the Hodge-Laplacian approach  (see \cite{AFW-Acta}).
We observe that the discretization that we are going to introduce for \eqref{K1_3-dadi}  will apply to \eqref{K1_3-Hgrad} as well}.
\end{remark}
\noindent %Given $\jj \in H_0(\div;\Om)$ with $\div\jj=0$, 
We first construct the interpolant $\jj_I\in V^{\face}_{\k}$ of $\jj$ that matches  the degrees of freedom \eqref{dof-3dfk-1G}:%--\eqref{dof-3dfk-2G}:
%\begin{align}
%& \bullet\quad \mbox{for each $f$: $\displaystyle\int_f
%((\jj-\jj_I)\cdot\nn) p_{k-1} \df=0 \quad \forall p_{k-1} \in \P_{k-1}(f) $} \label{dof-3dfk-1I}\\
%& \bullet\quad \mbox{for  $k\ge 2$, for each $\PP$,}\displaystyle\int_{\PP}
%(\jj-\jj_I)\cdot(\bgrad p_{k-1}) \dPP =0\quad \forall p_{k-1} \in \P_{k-1}(\PP) \label{dof-3dfk-2I} \\
%& \bullet\quad\mbox{for $ k\ge 1$, for each $\PP$,}\displaystyle\int_{\PP}
%(\jj-\jj_I)\cdot (\xx_{\PP}\wedge {\bf p}_{k-1}) \dPP =0\quad \forall p_{k-1} \in \P_{k-1}(\PP)
% \label{dof-3dfk-3I} 
%\end{align}
%
\begin{equation}
\bullet \mbox{ for each face $f$: $\displaystyle\int_f
(\jj-\jj_I)\cdot\nn_f  \df=0  $.} \label{dof-3dfk-1GI}
%& \bullet\mbox{ for each element $\PP$:}%\mbox{for } k\ge 1
%\displaystyle\int_{\PP}
%(\jj-\jj_I)\cdot (\xx_{\PP}\wedge {\pp}_{\k}) \dPP=0 \; \forall {\pp}_{\k} \in (\P_{\k}(\PP))^3.
% \label{dof-3dfk-3GI}
\end{equation}
From  the d.o.f. \eqref{dof-3dfk-1GI} and an integration by parts it follows that
\begin{equation}\label{commutface}
\int_{\PP} \div (\jj-\jj_I)\,\dPP=0\qquad {\forall\,\PP\in \Th }.
\end{equation}
Moreover, they also imply that $\jj_I\in H_0(\div;\Om)$ and that $\div\jj_I=0$ in $\Om$. Hence, according to \eqref{rot03fk}, we have that
$\jj_I$ will be the $\bcurl$ of some $\ww^*\in V^{\edge}_{\k}$:
\begin{equation}\label{avemariak}
\exists \ww^*\in V^{\edge}_{\k} \mbox{ such that } \bcurl\ww^*=\jj_I.
\end{equation}
Then we can introduce the {\bf discretization} of 
\eqref{K1_3}:
\begin{equation}\label{K1_3hk}
\left\{
\begin{aligned}
& \mbox {find  }\HH_h\in V^{\edge}_{\k} \mbox{ and }p_h\in V^{\nodal}_{1} \mbox{ such that: }\\
& [\bcurl\HH_h,\bcurl\vv]_{\face}+ [\nabla p_h,\mu\vv]_{\edge}=[\jj_I,\bcurl\vv]_{\face}
\quad\forall\vv\in V^{\edge}_{\k}\\
& [\nabla q,\mu\HH_h]_{\edge}=0\quad\forall q\in V^{\nodal}_{1}.\\ 
\end{aligned}
\right.
\end{equation}
We point out that both $\bcurl\HH_h$ and $\bcurl\vv$ (as well as $\jj_I$) are {\it face Virtual Elements in} $V^{\face}_{\k}(\PP)$ in each polyhedron $\PP$, so that
(taking also into account Remark \ref{catenak}) their {\it face} scalar products are computable as in \eqref{PSf3gk}.  Similarly, from the degrees of freedom of  a $q\in V^{\nodal}_{1}$ we can compute the degrees of freedom of
$\nabla q$, as an element of $V^{\edge}_{\k}$, so that the two edge-scalar products that appear in \eqref{K1_3hk} are computable as in \eqref{PSe3gk}. \\

\begin{prop}
Problem \eqref{K1_3hk} has a unique solution $(\HH_h,p_h)$, and $p_h\equiv 0$.
\end{prop}
\begin{proof}
Taking $\vv=\nabla p_h$ (as we did for the continuous problem \eqref{K1_3-dadi}) in the first equation, and using \eqref{SP3boundsk}  we easily obtain $p_h\equiv0$ for \eqref{K1_3hk} as well. To prove uniqueness of $\HH_h$, set $\jj_I=0$, and let ${\overline \HH}_h$ be the solution of the homogeneous problem. From the first equation we deduce that $\curl\,{\overline \HH}_h=0$. Hence, from \eqref{rot03k} we have ${\overline \HH}_h= \nabla q^*_h$ for some $q^*_h \in V^{\nodal}_{1}$. The second equation and \eqref{SP3boundsk} give then ${\overline \HH}_h=0$.
\end{proof}

 Once we know that $p_h=0$, the first equation of \eqref{K1_3hk} reads
 \begin{equation}
 [\bcurl\HH_h,\bcurl\vv]_{\face}=[\jj_I,\bcurl\vv]_{\face} \;\forall \vv\in V^{\edge}_{\k},
 \end{equation}
 that in view of \eqref{avemariak} becomes
  \begin{equation}
 [\bcurl\HH_h-\bcurl\ww^*,\bcurl\vv]_{\face}=0
 \quad\forall \vv\in V^{\edge}_{\k}.
 \end{equation}
 Using $\vv=\HH_h-\ww^*$ and \eqref{SPf3boundsk}, this easily implies
   \begin{equation}\label{bingo3k}
 \bcurl\HH_h=\bcurl\ww^*=\jj_I .
 \end{equation}

% ------------------------------------------------------------   
\subsection{Error estimates}\label{theo:est}
% ------------------------------------------------------------

For the theoretical derivations we consider the following mesh assumptions, that are quite standard in the VEM literature. We assume the existence of a positive constant $\gamma$ such that any polyhedron $\PP$  (of diameter $h_\PP$) satisfies the following conditions:
\begin{enumerate}
\item $\PP$ is star shaped with respect to a ball of radius bigger than $\gamma h_\PP$;
\item any face $f \in \partial\PP$ is star shaped with respect to a ball of radius bigger than 
$\gamma h_P$, and every edge of $\PP$ has length bigger than $\gamma h_\PP$. 
\end{enumerate}
We note that condition 1 (and 2) implies that $\PP$ (and any face of $\PP$) is simply connected.  At the theoretical level, some of the above conditions could be avoided by using more technical arguments. At the practical level, as shown by the numerical tests of the Section \ref{numres}, condition 2 is negligible since the method seems essentially impervious to degeneration of faces and edges. On the contrary, although the scheme is quite robust to distortion of the elements, condition 1 is more relevant since extremely anisotropic element shapes can lead to poor results. We also  recall that we assumed $\mu$ to be piecewise constant (see \eqref{micost}).

Let us bound the error $\HH-\HH_h$. We start by defining the interpolant $\HH_I\in V^{\edge}_{\k}$ of $\HH$, defined through the degrees of freedom \eqref{dof-3dek-1G}:%-\eqref{dof-3dek-5G}:
\begin{equation}
\bullet\mbox{ on each edge $e$: $\displaystyle\int_e (\HH-\HH_I)\cdot\tt_e \ds=0 . $ } \label{intHH1}
%& \bullet\mbox{ on each element $\PP$}:\displaystyle\int_{\PP}
%\bcurl(\HH-\HH_I)\cdot(\xx_{\PP}\wedge {\pp}_{\k})\,\dPP=0 \quad \forall {\pp}_{\k} \in (\P_{\k}(\PP))^3.  \label{intHH5}
\end{equation}

\begin{prop} With the choices  \eqref{dof-3dfk-1GI}
and  \eqref{intHH1}  we have
\begin{equation}\label{hailmaryk}
\bcurl\HH_I=\jj_I.
\end{equation}
\end{prop}

\begin{proof} From \eqref{rot03fk} we know that $\bcurl\HH_I \in V^{\face}_{\k}$. To prove \eqref{hailmaryk} we should just show that the {\it face degrees of freedom} \eqref{dof-3dfk-1G} of the difference  $\bcurl\HH_I - \jj_I$
are  zero, that is:
%\vfill\eject
\begin{equation}
 \mbox{$\forall f$: $\displaystyle\int_f
(\bcurl\HH_I -\jj_I)\cdot\nn_f \df =0 $.} \label{dof-3dfk-1Gp}
%& \bullet\,\forall\PP:\displaystyle\int_{\PP}
%(\bcurl\HH_I - \jj_I)\cdot (\xx_{\PP}\wedge {\pp}_{\k}) \dPP =0\quad \forall \pp_{\k}\! \in \!(\P_{\k}(\PP))^3.
% \label{dof-3dfk-3Gp}
\end{equation}
Since $\jj=\bcurl\HH$, from  the interpolation  formulas \eqref{dof-3dfk-1GI} we see that 
$$
\forall f:~\int_f (\bcurl \HH-\jj_I)\cdot\nn_f  \df=0,
$$
so that we can replace $\jj_I$ with $\bcurl\HH$ in \eqref{dof-3dfk-1Gp}, that becomes
\begin{equation}
  \forall f:~\displaystyle\int_f
\bcurl(\HH_I-\HH)\cdot\nn_f \, \df =0.  \label{ave1}
%& \bullet\displaystyle\int_{\PP}
%\bcurl(\HH_I-\HH)\cdot (\xx_{\PP}\wedge {\pp}_{\k}) \dPP =0\; \forall {\pp}_{\k} \in (\P_{\k}(\PP))^3. \label{ave3}
\end{equation}
Observing that \eqref{intHH1}  implies that
$$
\int_f
\rot_f(\HH-\HH_I)_{|f}\,  \df =0,
$$
and recalling that  on each $f$ the normal component of $\bcurl(\HH_I -\HH)$ is equal to the $\rot_f$ of the tangential components $(\HH_I-\HH)_{|f}$, we deduce
\begin{equation*}
\int_f \bcurl(\HH_I -\HH)\cdot\nn_f  \df\equiv\int_f\rot_f(\HH_I-\HH)_{|f}\,\df =0.
\end{equation*}
Hence, \eqref{ave1} is satisfied, and the proof is concluded.
\end{proof}
From \eqref{bingo3k} and \eqref{hailmaryk} it follows then
\begin{equation}\label{roteq3Dk}
\bcurl (\HH_I-\HH_h)=0
\end{equation}
and therefore, from \eqref{rot03k},
\begin{equation}\label{eungrad3Dk}
\HH_I-\HH_h= \nabla q^*_h \mbox{ for some }q^*_h\in V^{\nodal}_1 .
\end{equation}
Now we define an alternative $(L^2(\Om))^3$ inner product and norm that take into account the (piecewise constant) value of the permeability $\mu$. We set, for $\vv$ in $(L^2(\Om))^3$,
\begin{equation}\label{newnorm}
\tbn\vv\tbn_{0,\Om}^2
%\equiv\dba\vv,\vv\dbc_{0,\Om}
:=\int_{\Om}\mu\,|\vv|^2\dO.
\end{equation}
When $\vv=\HH$ in \eqref{newnorm} we get that $\tbn\HH\tbn_{0,\Om}^2=\int_{\Om}
\BB\cdot\HH\dO$, showing the connection between the new norm and the energy.
We now note that, using \eqref{SP3boundsk}, we have
\begin{equation}\label{stima03Dk}
 \alpha_*\tbn\HH_I-\HH_h\tbn_{0,\Om}^2\le [\HH_I-\HH_h, \mu(\HH_I-\HH_h)]_{\edge} ,
\end{equation}
and also (from \eqref{consiE3gk} and \eqref{micost})
\begin{equation}\label{consimu}
[\pp_0,\mu\vv]_{\edge}=\!
%\dba\pp_0,\vv\dbc_{0,\Om}
(\pp_0,\mu\vv)_{0,\Om}\quad\forall \vv\in V^{edge}_0 \mbox{ and }
\forall \pp_0 \mbox{ piecewise constant vector}.
\end{equation}
Then, starting with \eqref{stima03Dk} we have:
\begin{equation*}
%\label{stima1}
\begin{aligned}
 \alpha_*\tbn&\HH_I-\HH_h\tbn_{0,\Om}^2\le[\HH_I-\HH_h, \mu(\HH_I-\HH_h)]_{\edge} \quad\\[2mm]
&=\mbox{ (using \eqref{eungrad3Dk}) } [\HH_I-\HH_h, \mu\nabla q^*_h]_{\edge}\quad\\[2mm]
&=\mbox{ (using the second of \eqref{K1_3hk}) } [\HH_I, \mu\nabla q^*_h]_{\edge}\quad\\[2mm]
&=\mbox{ (adding and subtracting  $\Pi_{\k}\HH$) } [\HH_I-\Pi_{\k}\HH, \mu\nabla q^*_h]_{\edge}+[\Pi_{\k}\HH, \mu\nabla q^*_h]_{\edge}\qquad\\[2mm]
&= \mbox{(using \eqref{consimu}) }[\HH_I-\Pi_{\k}\HH, \mu\nabla q^*_h]_{\edge}+(\Pi_{\k}\HH, \mu \nabla q^*_h)_{0,\Om}\quad\\[2mm]
&=\mbox{ (adding and subtracting  $\HH$) }[\HH_I-\Pi_{\k}\HH, \mu\nabla q^*_h]_{\edge}+(\Pi_{\k}\HH-\HH, \mu \nabla q^*_h)_{0,\Om}
+ (\HH, \mu \nabla q^*_h)_{0,\Om}\quad\\[2mm]
&=\mbox{ (from the second of \eqref{K1_3}) }[\HH_I-\Pi_{\k}\HH,\mu \nabla q^*_h]_{\edge}+(\Pi_{\k}\HH-\HH, \mu \nabla q^*_h)_{0,\Om}
\quad\\[2mm]
&\le\mbox{ (using Cauchy-Schwarz and \eqref{SP3boundsk}) }\Big( \alpha^* \,\tbn\HH_I-\Pi_{\k}\HH\tbn_{0,\Om}+  \tbn\Pi_{\k}\HH-\HH\tbn_{0,\Om}\Big)\,\tbn\nabla q^*_h\tbn_{0,\Om}\quad\\[2mm]
&\le \mbox{ (using again \eqref{eungrad3Dk}) }\Big( \alpha^*\,\tbn\HH_I-\Pi_{\k}\HH\tbn_{0,\Om}+  \tbn\Pi_{\k}\HH-\HH\tbn_{0,\Om}\Big)\,\tbn\HH_I-\HH_h\tbn_{0,\Om}
\end{aligned}
\end{equation*}
that implies immediately (since $\alpha^* \ge 1$)
\begin{equation}
\tbn\HH_I-\HH_h\tbn_{0,\Om}\le\frac{\alpha^*}{\alpha_*}(\tbn\HH_I-\Pi_{\k}\HH\tbn_{0,\Om}+\tbn\Pi_{\k}\HH-\HH\tbn_{0,\Om}).
\end{equation}
We can summarize the result in the following theorem.
\begin{thm}
Problem \eqref{K1_3hk} has a unique solution and the following estimate holds:
\begin{equation}\label{errorL2}
\tbn\HH-\HH_h\tbn_{0,\Om}\le C\,\Big(\tbn\HH-\HH_I\tbn_{0,\Om}+\tbn\HH-\Pi_{\k}\HH\tbn_{0,\Om}\Big),
\end{equation}
with $C$ a constant independent of the mesh size.
Moreover, thanks to \eqref{bingo3k} we also have
\begin{equation}
\|\bcurl(\HH-\HH_h)\|_{0,\Om}=\|\jj-\jj_I\|_{0,\Om}.
\end{equation}
%so that
%\begin{equation}\label{errorCurl}
%\|\HH-\HH_h\|_{H(\bcurl;\Om)}\le C\,\Big(\|\HH-\HH_I\|_{0,\Om}+\|\HH-\Pi_{\k}\HH\|_{0,\Om}+\| \jj-\jj_I\|_{0,\Om}\Big).
%\end{equation}
\end{thm}
The above result can be combined with standard polynomial approximation estimates on  star shaped polyhedra in order to estimate the terms involving the $L^2$ projection on polynomials. 
{Moreover, approximation estimates for the VEM interpolants $\HH_I$ and $\jj_I$ can be derived (under the mesh assumptions at the beginning of this section) by an extension of the arguments in 
\cite{Acoustic}, \cite{BLR-stab}, \cite{Brenner}, \cite{Steklov-VEM}.}
 We therefore obtain, provided that $\HH$ and $\jj$ are sufficiently regular,
$$
\tbn\HH-\HH_h\tbn_{0,\Om}\le C h 
%|\mu^{1/2}\HH|_{1,\Om}  , 
\qquad 
\|\bcurl(\HH-\HH_h)\|_{0,\Om} \le C h. 
%|\jj|_{1,\Om} .
$$

\section{Numerical results}
\newcommand{\normConsideredForPh}{$l^\infty-$norm~}
\newcommand{\FLUX}{\textsl{FLUX2D}~}
\label{numres}
%---------------------------------------------------------------------------------
In this section we provide some numerical results.
{We first provide a test on a polyhedral domain with Dirichlet  boundary data, then a test on a cylindrical domain with jumping coefficients and Neumann boundary data, and finally a benchmark example with a more complex geometry.}

\subsection{Test case 1: $h$-analysis with homogeneous Dirichlet boundary conditions}

In this subsection we set $\mu=1$ and take as exact solution of~ \eqref{Max3}
the vector field  
$$
\HH(x,\,y,\,z) := \frac{1}{\pi}	\left(\begin{array}{r}
				\sin(\pi y) - \sin(\pi z)\\
				\sin(\pi z) - \sin(\pi x)\\
				\sin(\pi x) - \sin(\pi y)
				\end{array}\right)\,.
$$
We consider as domain $\Omega$ the truncated octahedron~\cite{patagonGeo} and three different  discretizations, 
see Fig.~\ref{fig:meshesTO}:
\begin{itemize}
 \item \textbf{Structured}, a mesh composed by structured cubes inside $\Omega$ and 
 arbitrarily shaped elements close to the boundary;
 \item \textbf{CVT}, a Centroidal Voronoi Tessellation of $\Omega$ obtained 
 via a standard Lloyd algorithm~\cite{cvtPaper};
 \item \textbf{Random}, a mesh obtained 
 by the constrained Voronoi Tessellation of points randomly put inside  $\Omega$.
\end{itemize}
{
The first type of mesh can be generated by using a Voronoi generation algorithm with seeds disposed in a regular formation, see \cite{apollo11}. This approach allows to approximate complex geometries and still inherits many among the computational advantages of regular cubical meshes.
The last type of meshes is instead interesting in order to check the robustness of the method, since it exhibits edges and faces (see the details in Fig.~\ref{fig:meshesTO}) that can be very small with respect to the size of the parent element.}
{To get such discretizations, we use the c++ library \texttt{voro++}~\cite{voroPlusPlus}.}

\begin{figure}[!htb]
\begin{center}
\begin{tabular}{ccc}
Structured &CVT &Random\\\\ 
\includegraphics[width=0.31\textwidth]{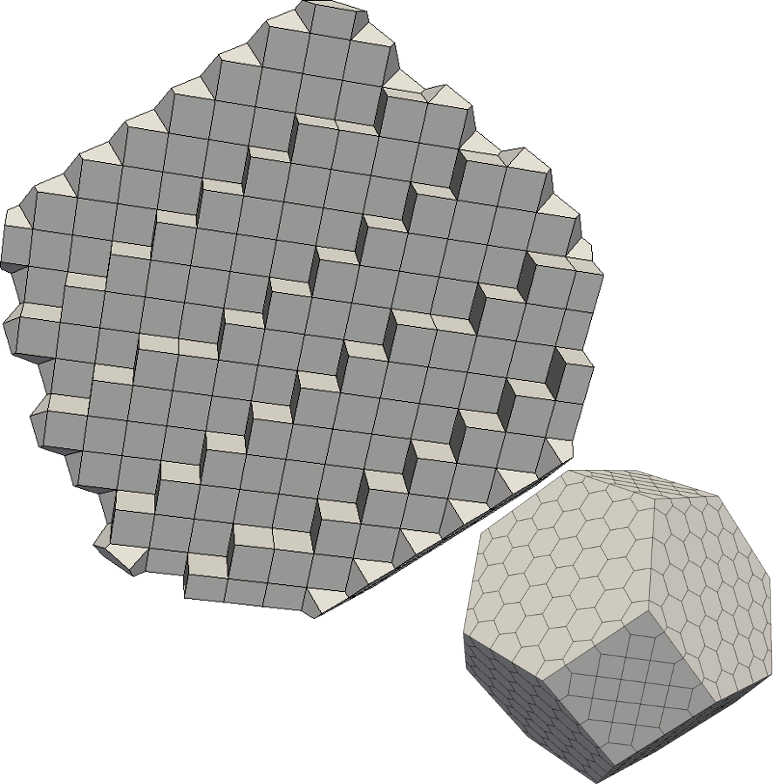} &
\includegraphics[width=0.31\textwidth]{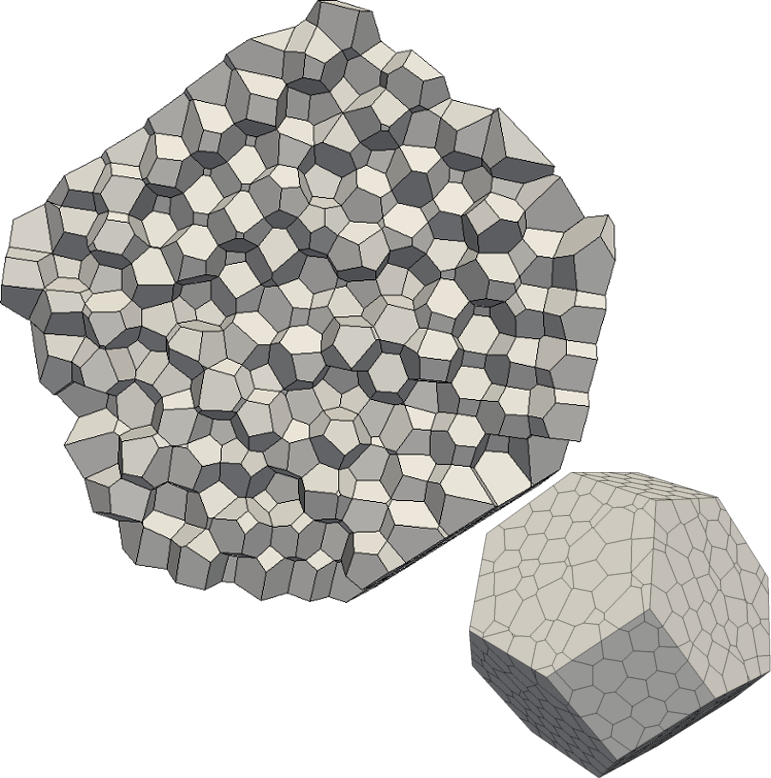}	  &
\includegraphics[width=0.31\textwidth]{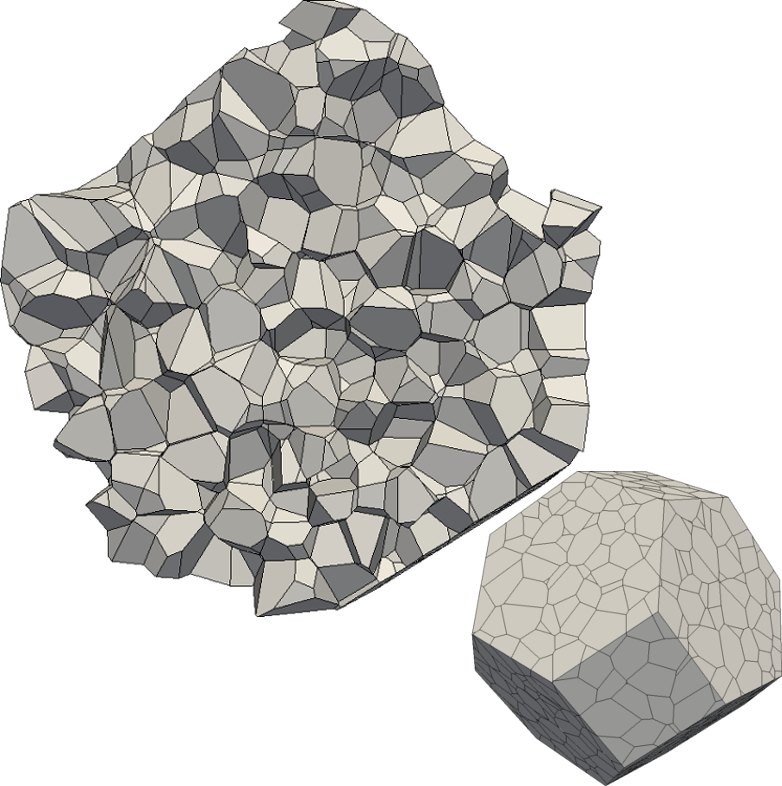} \\
\end{tabular}
\end{center}
\caption{Three different discretizations of the truncated octahedron.}
\label{fig:meshesTO}
\end{figure}
We associate to each mesh a mesh-size $h$ defined as
$$
h:= \frac{1}{N_P}\sum_{i=1}^{N_P} h_P\,,
$$
where $N_P$ is the number of polyhedrons in the mesh and
$h_P$ is the diameter of the polyhedron $P$. 
We compute the $L^2$-error on $\HH$ as
$$
\frac{||\HH - \Pi_0\HH_h||_{0,\,\Omega}}{||\HH||_{0,\,\Omega}}\,,
$$
where $\Pi_0\HH_h$ is the piecewise constant projection of $\HH_h$ defined in \eqref{proj-edge-3D}.
Fig.~\ref{fig:convEse1} (left) shows the convergence rate for each type of mesh.
The slopes are coherent with the theory, see Equation~\eqref{errorL2}.
Moreover, from this graph we can better appreciate the robustness of the VEM with respect to element distortions. 
Indeed, the convergence lines for the three meshes are very close to each other.

\begin{figure}[!htb]
\begin{minipage}{0.5\textwidth}
\begin{center}
\includegraphics[width=1.\textwidth]{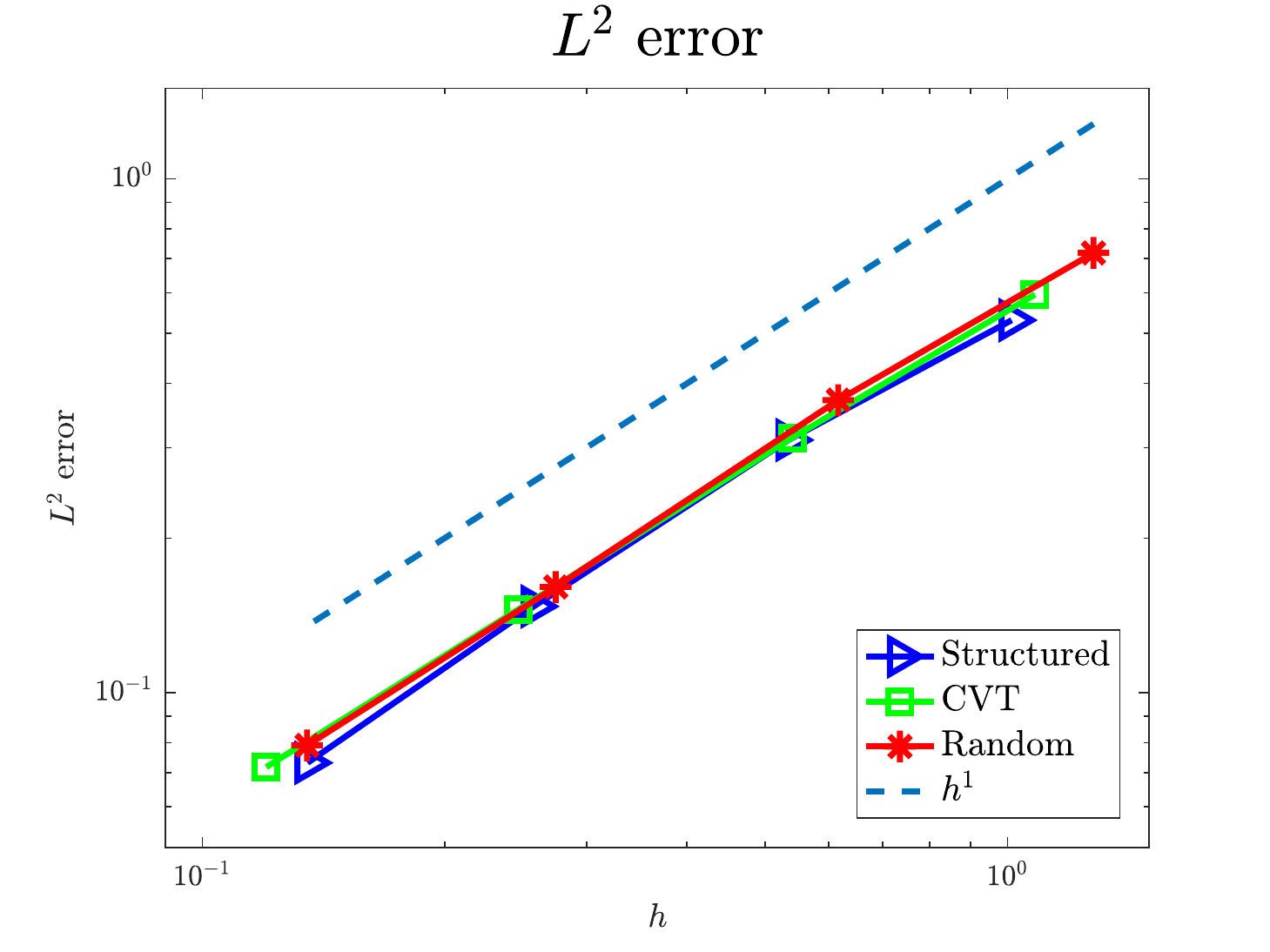}
\end{center}
\end{minipage}%
\begin{minipage}{0.5\textwidth}
\begin{center}
\begin{tabular}{|c|c|c|c|}
\multicolumn{4}{c}{$l^\infty-$norm}\\
\hline
step &Structured &CVT &Random \\
\hline
1 &1.5098e-15 &1.1844e-15 &4.7323e-13 \\ 
2 &7.0101e-16 &2.5902e-14 &1.6107e-12 \\ 
3 &2.6762e-15 &1.0476e-13 &1.8733e-10 \\ 
4 &7.0545e-15 &1.0953e-10 &1.0001e-07 \\ 
\hline
\end{tabular}
\end{center}
\end{minipage}
\caption{Test case 1:  $L^2-$error for $\HH$ (left), and 
$l^\infty-$norm of $dof(p_h)$ (right).}
\label{fig:convEse1}
\end{figure}

{
The results for the Lagrange multiplier  $p_h$ are shown in Fig.~\ref{fig:convEse1} (right). 
As shown in Section~\ref{MS},  the exact solution  $p_h$ of the discrete problem is identically zero. Clearly, the roundoff errors generate, out of the computer, a $p_h$ that is {\it almost} identically zero. In some sense we could then take the value of the {\it computed} $p_h$ as a measure (or, better, a rough indicator) of the conditioning of the final linear system. In particular we see that the Random meshes  generate a worse conditioning.
%%
%This observation motivates also the fact that the Random meshes provide the worst results.
}
\subsection{Test case 2: discontinuous $\mu$ and homogeneous Neumann conditions}

In this subsection we consider an example taken from~\cite{bermudez2008finite}.
The geometry consists in a cylindrical domain with two concentric connected components, 
$\mathcal{S}_1$ and $\mathcal{S}_2$,
separated by a magnetic material, $\mathcal{M}$ (see a cross-section in Fig.~\ref{fig:currentSourgeScheme}).
A current of the same intensity $I=$\num{70000}A
but opposite direction passes along $\mathcal{S}_1$ and 
$\mathcal{S}_2$.
% as shown in the cross-section depicted in Fig.~\ref{fig:currentSourgeScheme}.
%
\begin{figure}[!htb]
\begin{center}
\hspace{2cm}\includegraphics[width=0.4\textwidth]{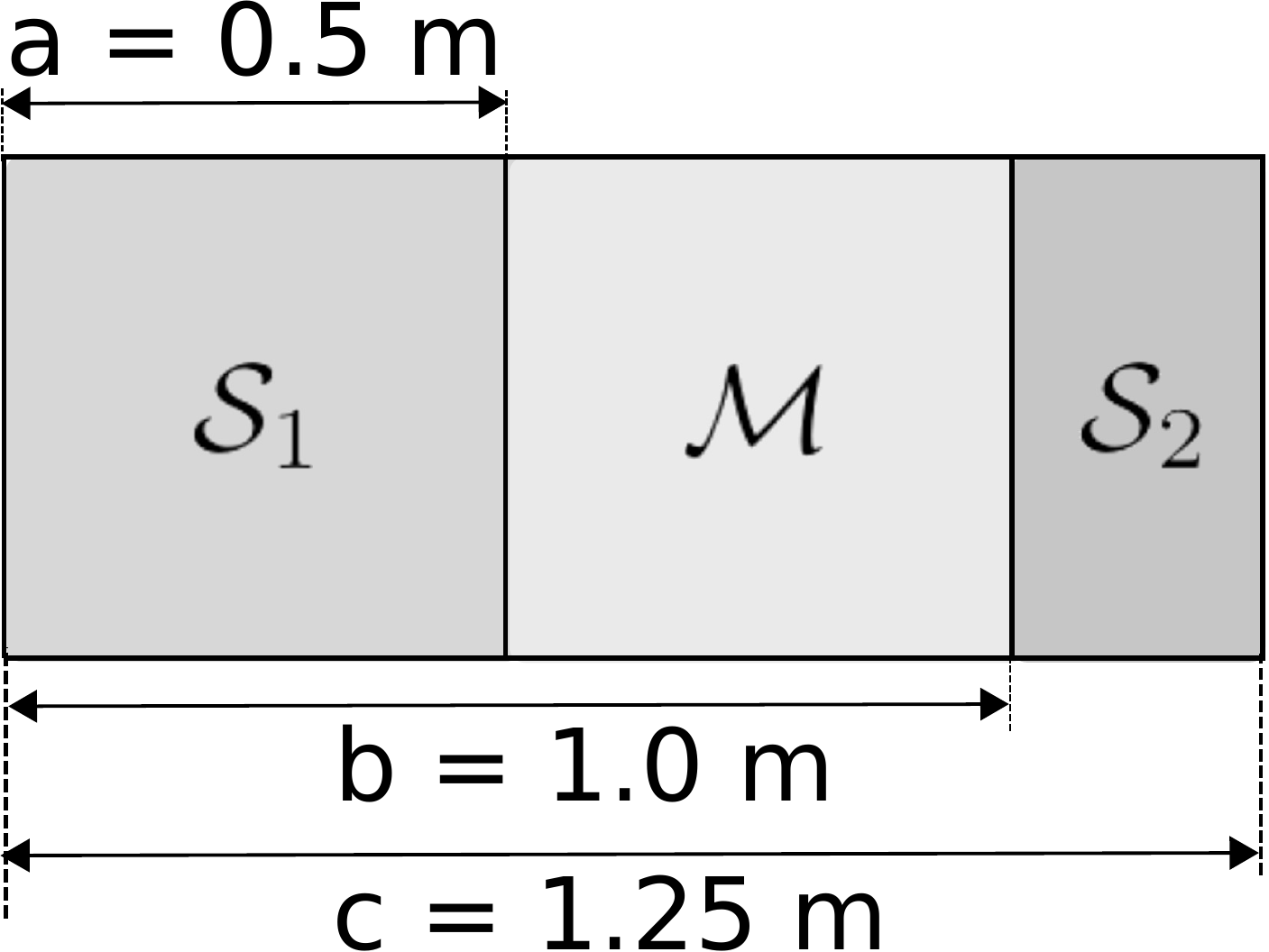}
\end{center}
\caption{Test case 2: meridian cross section of the computational domain.}
\label{fig:currentSourgeScheme}
\end{figure}
We assign the following current intensity (with $\mathbf{e}_z:=(0,\,0,\,1)^t$)
\begin{equation}
\jj(x,\,y,\,z) := \begin{cases}
		  \mathlarger{\frac{I}{\pi a^2}}\,\mathbf{e}_z         &\text{if }\xx\in\mathcal{S}_1\\
		  \mathbf{0}                                           &\text{if }\xx\in\mathcal{M}\\
		  \mathlarger{-\frac{I}{\pi (c^2-b^2)}}\,\mathbf{e}_z  &\text{if }\xx\in\mathcal{S}_2.\\		  
		  \end{cases}
\end{equation}
 We apply homogeneous Neumann boundary conditions and we set
  $\mu_{\mathcal{S}_1}=\mu_{\mathcal{S}_2}=1.0$, $\mu_{\mathcal{M}}=1000.0$. Then
the exact solution is given by 
\begin{equation}
\HH(x,\,y,\,z) := \begin{cases}
		  \mathlarger{\frac{I}{2\pi a^2}}\,\mathbf{e}_\theta                        
		  &\text{if }\xx\in\mathcal{S}_1\\[10pt]
		  \mathlarger{\frac{I}{2\pi r^2}}\,\mathbf{e}_\theta                        
		  &\text{if }\xx\in\mathcal{M}\\[10pt]
		  \left[\mathlarger{-\frac{I}{2\pi(c^2-b^2)}+\frac{1}{r^2}\left(\frac{I}{2\pi}+\frac{Ib^2}{2\pi(c^2-b^2)}\right)}\right]\,\mathbf{e}_\theta
		  &\text{if }\xx\in\mathcal{S}_2\\[10pt]
		  \end{cases}
\end{equation}
where $\mathbf{e}_\theta:=(-y,\,x,\,0)^t$, $r=\sqrt{x^2+y^2}$ and $a=0.5, b=1, c= 1.25$ (see Fig. \ref{fig:currentSourgeScheme}).

To deal with Neumann boundary conditions, problem \eqref{K1_3-dadi} must be modified, looking for $\HH\in H(\bcurl;\Om)$ and $p\in H^1(\Om)$. 
The magnetic field $\HH$ will still be unique, while $p$ will be determined only up to a constant. 
In the code the constant is fixed by requiring the average of the vertex values of $p_h$ to be zero.

%{\color{blue}
%Since the magnetic permeability $\mu$ is not constant,
%this example represents also a numerical validation that
%the theory developed in this article can handle a discontinuous magnetic permeability.}
%
We build two meshes of the cylinder
by extruding two planar two-dimensional  meshes: one made of polygons~\cite{PolyMesher} and 
one made of triangles~\cite{shewchuk1996triangle}.
We refer to the former as voro-extrusion and to the latter as tria-extrusion (for an example of both see Fig.~\ref{fig:meshBer61}).
Then we make two sets of meshes with decreasing mesh-size.
%to numerically verify the trend of the solution.
%In Fig.~\ref{fig:meshBer61}, we show both voro-extrusion and tria-extrusion meshes 
%we use to make the convergence analysis.}
%
\begin{figure}[!htb]
\begin{center}
\begin{tabular}{ccc}
\includegraphics[width=0.4\textwidth]{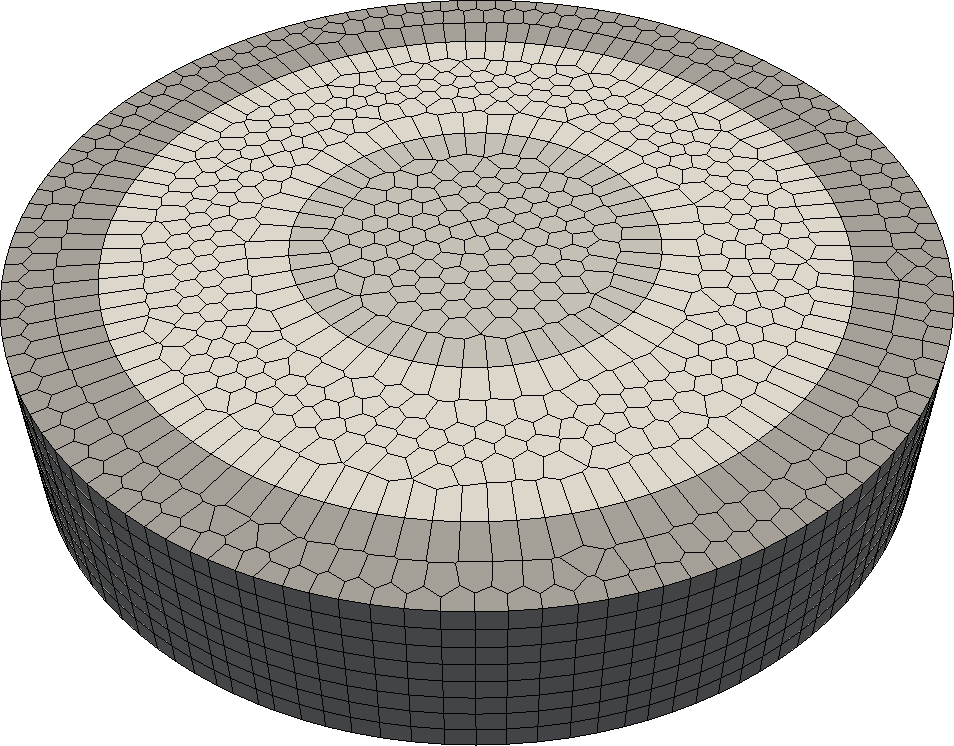} &\phantom{mm} &  
\includegraphics[width=0.4\textwidth]{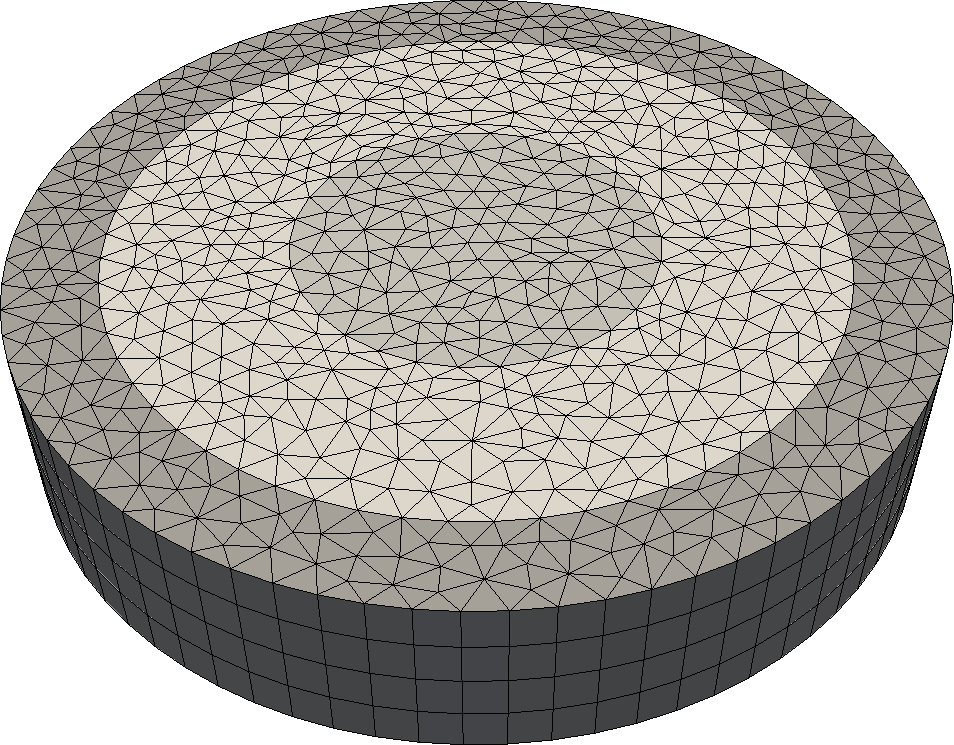}
\end{tabular}
\end{center}
\caption{Test case 2: two polyhedral decompositions  of the cylinder:
voro-extrusion (left), and tria-extrusion (right).}
\label{fig:meshBer61}
\end{figure}
In Fig.~\ref{fig:convEse2} we depict the convergence curves on the left, and on the right
%and in Fig.~\ref{fig:convEse2} right 
we provide the \normConsideredForPh of the vectors $dof(p_h)$.
Both quantities behave as expected:
%for both mesh types.
indeed, we get a convergence rate equal to 1, for the $L^2$ error on the vector field $\HH_h$,
while the values of \normConsideredForPh vanish up to machine algebra errors.
\begin{figure}[!htb]
\begin{minipage}{0.5\textwidth}
\begin{center}
\includegraphics[width=1.\textwidth]{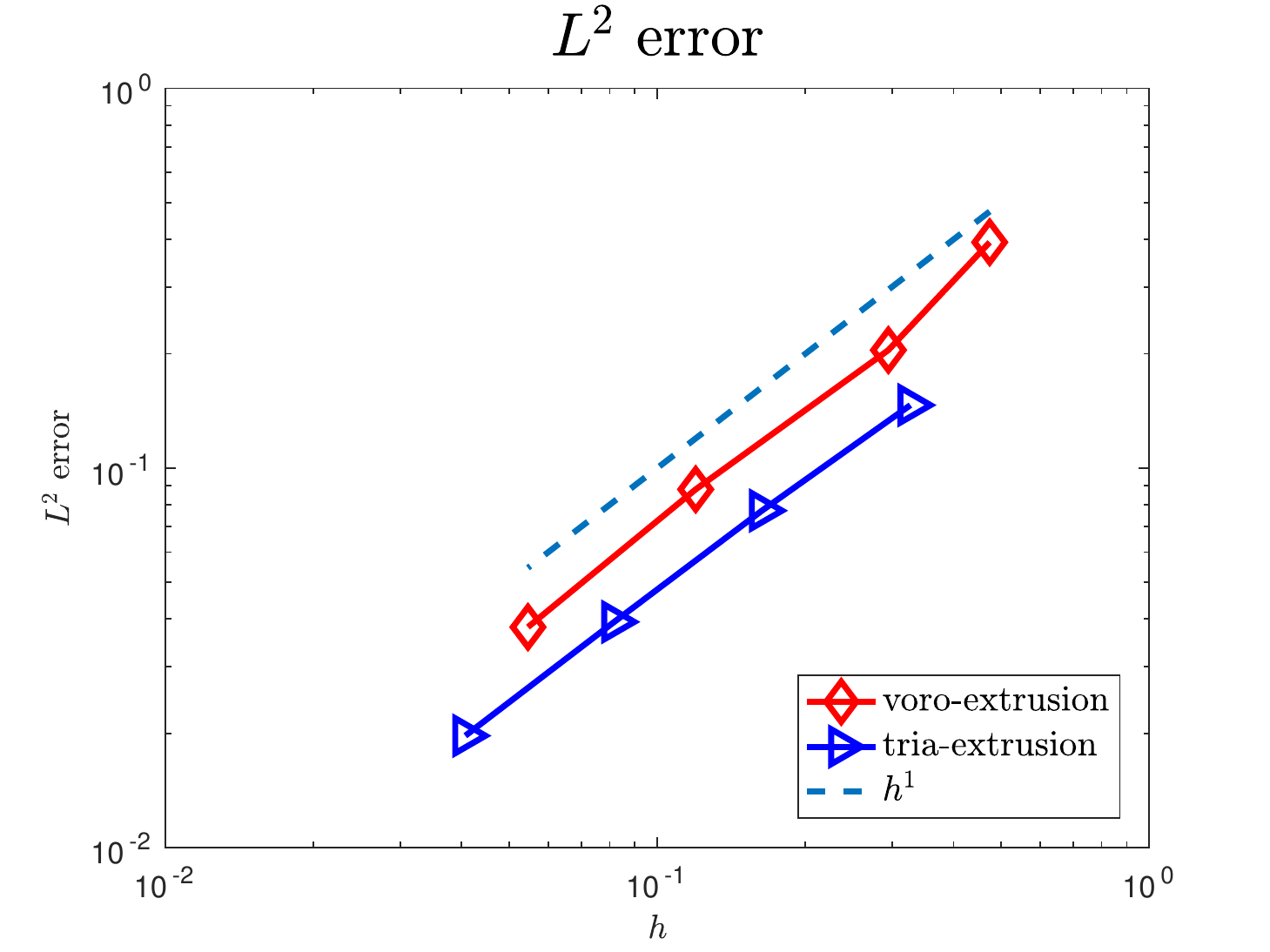}
\end{center}
\end{minipage}%
\begin{minipage}{0.5\textwidth}
\begin{center}
\begin{tabular}{|c|c|c|}
\cline{2-3}
\multicolumn{1}{c|}{}&\multicolumn{2}{c|}{$l^\infty-$norm}\\
\hline
step &voro-extrusion &tria-extrusion \\
\hline
1 &2.1273e-11 &1.3632e-10 \\                                                               
2 &1.0564e-10 &4.1360e-10 \\
3 &1.4887e-10 &1.0027e-09 \\
4 &4.5312e-10 &1.1151e-08 \\
\hline
\end{tabular}
\end{center}
\end{minipage}
\caption{Test case 2:   $L^2$ error for $\HH$ (left), and
$l^\infty-$norm of  $dof(p_h)$ (right).}
\label{fig:convEse2}
\end{figure}
\subsection{Test case 3: a cylindrical electromagnet}

In this subsection we consider a typical benchmark problem,
see e.g. ~\cite{Bermudez-Book,bermudez2008finite,centeradaptive}.
The geometry consists in a ferromagnetic cylindrical core, $\mathcal{C}$, 
surrounded by a toroidal coil with a rectangular cross section, $\mathcal{T}$,
with air, $\mathcal{A}$, around these two structures.
In Fig.~\ref{fig:bermudezScheme} we show a meridian cross section of the domain where 
we specify the dimensions of the cylindrical core, the toroidal coil and the bounding box of the domain.

\begin{figure}[!htb]
\begin{center}
\includegraphics[width=0.35\textwidth]{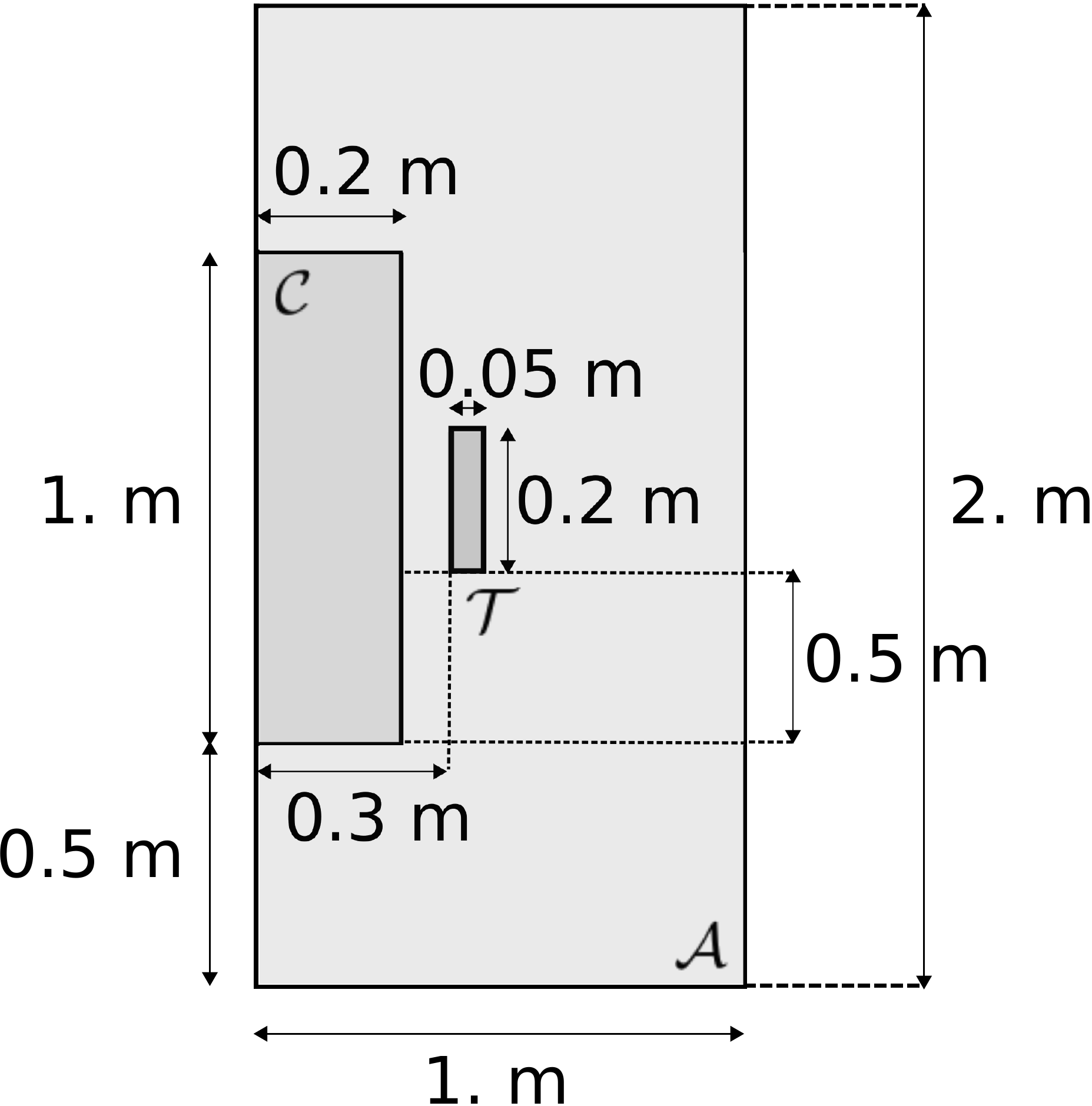}
\end{center}
\caption{Test case 3: meridian cross section of the computational domain.}
\label{fig:bermudezScheme}
\end{figure}

One looks for the magnetic flux generated by a constant current passing along the toroidal coil.
More specifically, we assume  a constant current of 1 Amp\`ere, i.e., 
$$
\jj(x,\,y,\,z) := \begin{cases}
		  \mathlarger{\frac{1}{A|\mathbf{e}_\theta|}}\,\mathbf{e}_\theta     &\text{if }\xx\in\mathcal{T}\\
		  \mathbf{0}                                                         &\text{otherwise}\\
		  \end{cases}
$$
where $\mathbf{e}_\theta=\left(-y,\,x,\,0\right)^t\,$,
$|\mathbf{e}_\theta|$ denotes the norm of $\mathbf{e}_\theta$
and $A$ is the area of the cross section of the coil.
The relative magnetic permeability is taken as:
$\mu_{\mathcal{T}}=1.0$ for the coil,  
$\mu_{\mathcal{C}}=10000.0$ in the  ferromagnetic core, and 
$\mu_{\mathcal{A}}=1.0$ for the air.
Moreover, we suppose that the artificial boundary is sufficiently far 
from both the ferromagnetic cylinder and the toroidal core
so that we can apply the Neumann boundary conditions $\HH\cdot\nn = 0$.

In Fig.~\ref{fig:meshBer62} we show one of the meshes used in this example.
Here too the  meshes  are constructed by extruding two-dimensional ones.
\begin{figure}[!htb]
\begin{center}
\begin{tabular}{ccc}
\includegraphics[height=0.3\textheight]{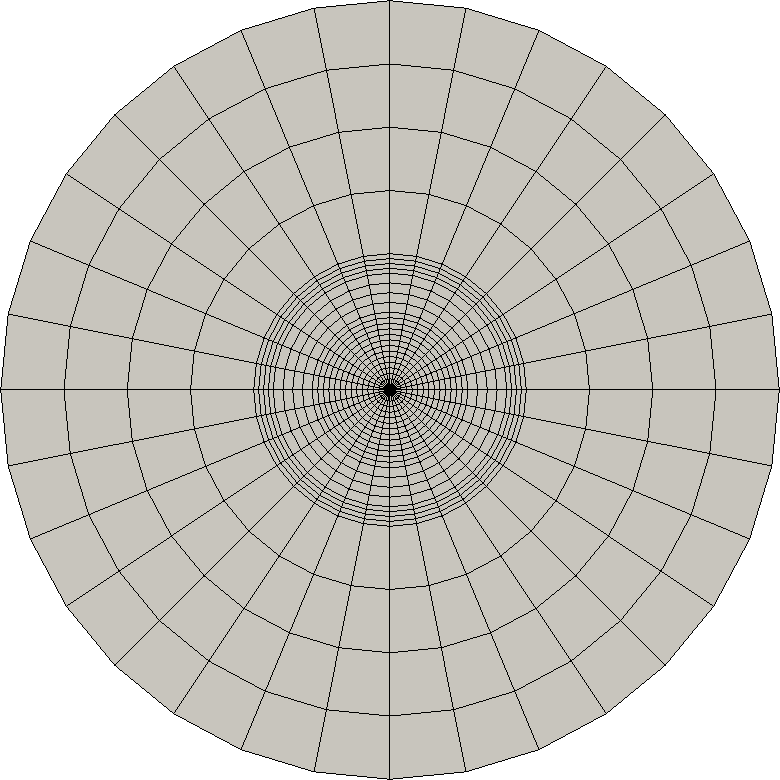}&\phantom{mmm}&
\includegraphics[height=0.3\textheight]{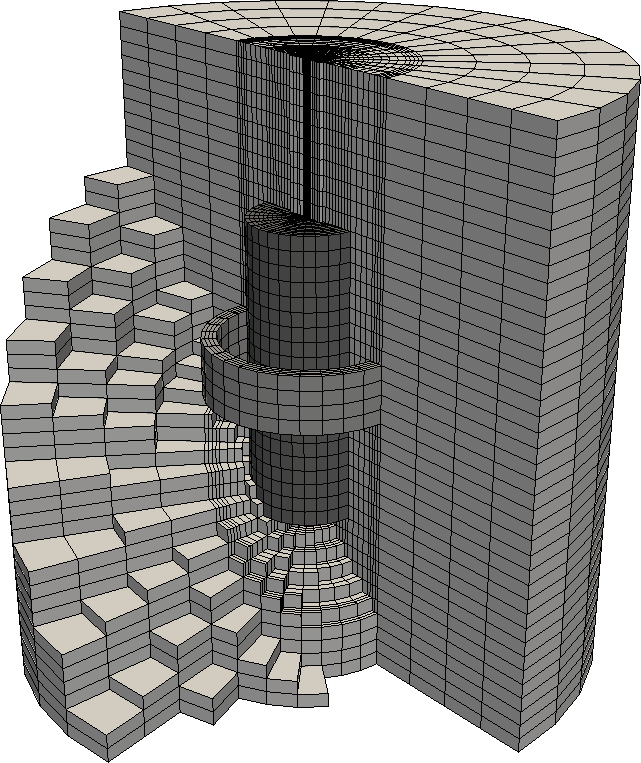}\\
\end{tabular}
\end{center}
\caption{Test case 3:
On the right a clip of the 3d mesh obtained extruding the 2d mesh on the left.}
\label{fig:meshBer62}
\end{figure}
Figs.~\ref{fig:fieldBermudez1} and~\ref{fig:fieldBermudez2} show a qualitative comparison between the 
solution provided in~\cite{Bermudez-Book} (left),
and that obtained by the present method (right).
More specifically, in Fig.~\ref{fig:fieldBermudez1},
we show the modulus of the magnetic flux density, i.e. the modulus of  $\BB=\mu \HH$, inside the cylindrical core. Then, in Fig.~\ref{fig:fieldBermudez2} 
we show $\BB$ along some cross sections of the cylindrical core and the toroidal coil.
In both cases the results of two methods show a comparable behavior.
\begin{figure}[!htb]
\begin{center}
\begin{tabular}{cccc}
\includegraphics[width=0.2\textwidth,height=0.4\textheight]{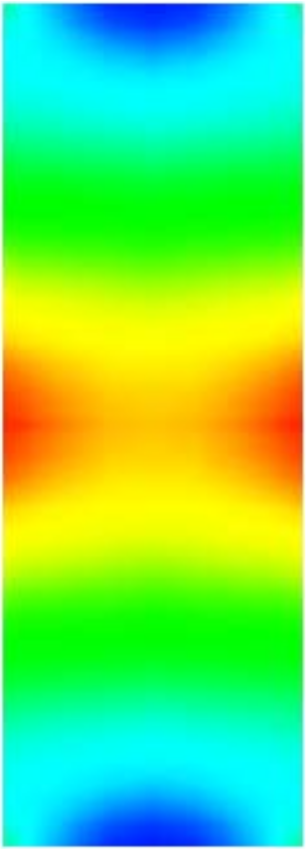}&\phantom{mmmmm}&
\includegraphics[width=0.2\textwidth,height=0.4\textheight]{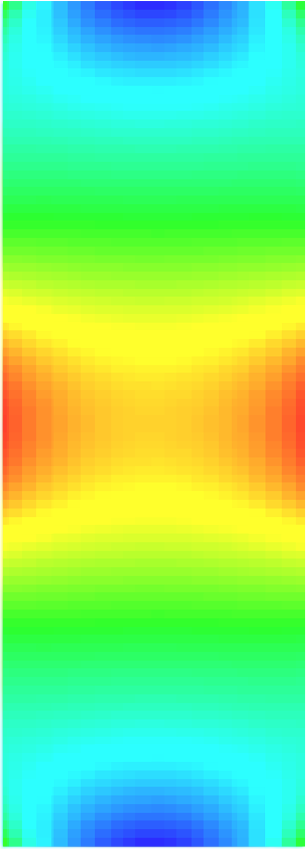}&
\includegraphics[width=0.2\textwidth]{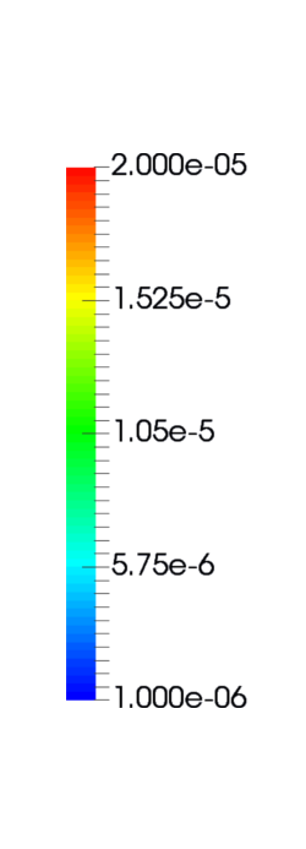}\\
\end{tabular}
\end{center}
\caption{Test case 3: $|\BB|$ in a section of ${\mathcal C}$: from~\cite{Bermudez-Book} (left)
and that from the present approach (right).}
\label{fig:fieldBermudez1}
\end{figure}
\begin{figure}[!htb]
\begin{center}
\begin{tabular}{cccc}
\includegraphics[width=0.3\textwidth]{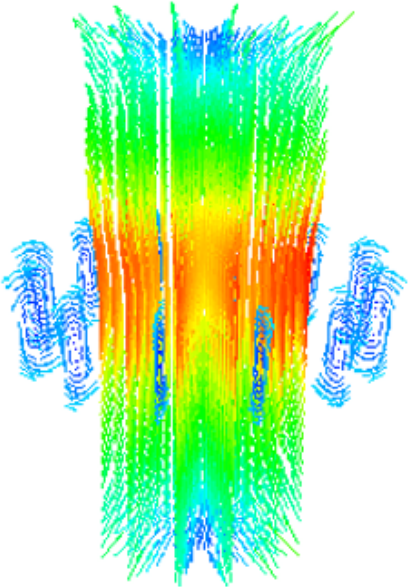}&\phantom{mm}&
\includegraphics[width=0.3\textwidth]{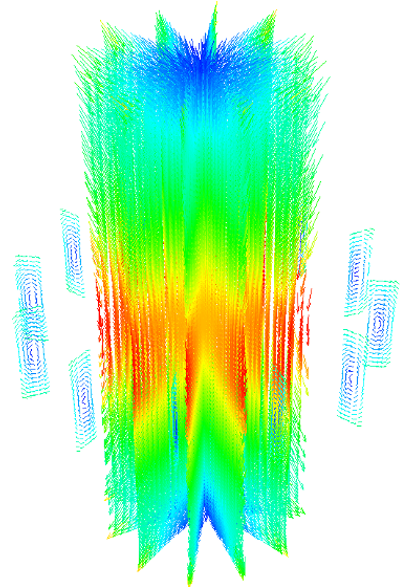}&
\includegraphics[width=0.2\textwidth]{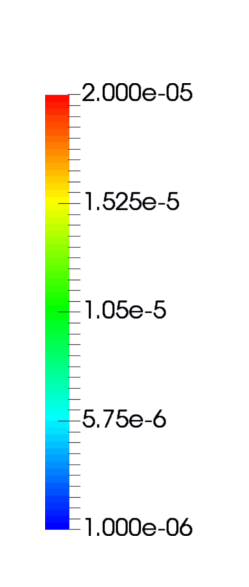}
\end{tabular}
\end{center}
\caption{Test case 3: $\BB$ in a section of  ${\mathcal C}$ and ${\mathcal T}$:
 from~\cite{Bermudez-Book} (left), and  from the present approach (right).}
\label{fig:fieldBermudez2}
\end{figure}

\bigskip

Finally, to have a more quantitative validation of this example,
we compute the so-called magnetic energy 
$$
W:=\int_{\mathcal{D}} \BB\cdot\HH\,\dO=\int_{\mathcal{D}} \mu|\HH|^2\,\dO\,,
$$
for each sub-domain $\mathcal{C}$, $\mathcal{T}$ and $\mathcal{A}$ (see also \eqref{newnorm}).
We consider a sequence of three nested meshes  to verify the convergence rate of these energies. We refer to these meshes as mesh1, mesh2 and mesh3,
the first mesh being the coarsest.

Since we do not have the exact solution of this problem, in order to assess  the accuracy of our results for the magnetic energy $W$
we proceed as in \cite{bermudez2008finite}, i.e.
we consider as exact values of the magnetic energies the values obtained by the software \FLUX on a very fine mesh. The numerical solution of \FLUX follows a scalar potential formulation and exploits the symmetry of the domain via a two-dimensional cylindric coordinate system.
The data are collected in Table~\ref{tab:W} and 
in parenthesis we show the relative error.
The method behaves as expected.

%  and 
%the values of $W$ on each sub-domain converge to the values provided by \FLUX 
%(the relative error is decreasing with the size of the mesh).

\begin{table}[!htb]
\begin{center}
\begin{tabular}{|c|cr|cr|cr|}
\hline
$W$ in                                  &\multicolumn{2}{c|}{$\mathcal{A}$}  
                                        &\multicolumn{2}{c|}{$\mathcal{C}$}
                                        &\multicolumn{2}{c|}{$\mathcal{T}$}\\
\hline
\FLUX	 	  	  	        &\multicolumn{1}{c}{9.09e-07} 
					&\multicolumn{1}{c|}{}
					&\multicolumn{1}{c}{4.73e-10}
					&\multicolumn{1}{c|}{}
					&\multicolumn{1}{c}{3.61e-08}          
					&\multicolumn{1}{c|}{}\\ 
\hline
\hline
mesh1 VEM				&9.70e-07 &(6.7\%)  &7.54e-10 &(59.4\%)  &3.29e-08 &(8.8\%)   \\
mesh2 VEM			        &9.22e-07 &(1.4\%)  &5.53e-10 &(16.9\%)  &3.81e-08 &(5.5\%)   \\
mesh3 VEM			        &9.11e-07 &(0.2\%)  &4.98e-10 &(5.2\%)   &3.75e-08 &(3.8\%)   \\
\hline                                                                                     
\end{tabular}
\end{center}
\caption{Test case 3: Behavior of $W$, in the three regions, computed with the present method}
\label{tab:W}
\end{table}

\section*{Acknowledgments}

The first and third authors were partially supported by the European Research Council through the H2020 Consolidator Grant (grant no. 681162) CAVE -- Challenges and Advancements in Virtual Elements. This support is gratefully acknowledged. 

\bigskip

\bibliographystyle{amsplain}

\bibliography{general-bibliography}

\end{document}